\documentclass[11pt]{amsart}
\usepackage{mathrsfs}
\usepackage{amsfonts}
\usepackage{latexsym,amsmath,amssymb}
\usepackage[colorlinks=true,linkcolor=blue]{hyperref}
\usepackage{indentfirst}
\usepackage{cases}
\usepackage{amssymb}
\usepackage[numbers,sort&compress]{natbib}

 \renewcommand{\epsilon}{\varepsilon}

\newtheorem{theorem}{Theorem}[section]

 \newtheorem{lemma}[theorem]{Lemma}
 \newtheorem{remark}[theorem]{Remark}

 \newtheorem{Corollary}[theorem]{Corollary}
 \newtheorem{proposition}[theorem]{proposition}
 
\newtheorem{deff}[theorem]{Definition}
 
 \newcommand{\bth}{\begin{theorem}}
 \newcommand{\ble}{\begin{lemma}}
 \newcommand{\bcor}{\begin{corr}}
 \newcommand{\bdeff}{\begin{deff}}
 \newcommand{\bprop}{\begin{proposition}}
 \newcommand{\ele}{\end{lemma}}
 \newcommand{\ecor}{\end{corr}}
 \newcommand{\edeff}{\end{deff}}
 
 \newcommand{\eprop}{\end{proposition}}

 \renewcommand{\Pi}{\varPi}

 \renewcommand{\epsilon}{\varepsilon}

  \newcommand{\R}{{\mathbb R}}

\numberwithin{equation}{section}

\title
[Nonlinear second boundary conditions ]{The mean curvature type hypersurfaces
with prescribed gradient image}

 \author{Jiguang Bao}
 
\address{School of Mathematical Sciences, Beijing Normal University,
	Beijing, 100875, People's Republic of China,
	E-mail: jgbao@bnu.edu.cn,
	ORCID: 0000-0001-5545-1560}

\author{Rongli Huang   }

\address{School of Mathematics and Statistics, Guangxi Normal University,
Guilin, Guangxi 541004, People's Republic of China,
 E-mail: ronglihuangmath@gxnu.edu.cn }

 \author{Qinfeng Jiang*}
\address{School of Mathematical Sciences, Beijing Normal University,
	Beijing, 100875, People's Republic of China,
	E-mail: 202321130046@mail.bnu.edu.cn,
	ORCID: 0009-0003-0879-073X}
 \thanks{*Corresponding author.}
 
 \subjclass[2020]{Primary 35J25; Secondary 53A10}
 \keywords{mean curvature type equation, second boundary value problem,  Gauss map, Klein ball}

\begin{document}
\begin{abstract}
In this paper, we consider the existence of mean curvature type hypersurfaces
 with prescribed gradient image.  Let $\Omega$ and $\tilde{\Omega}$ be uniformly convex
 bounded domains in $\mathbb{R}^n$ with smooth boundary. We show that there exists unique convex solutions  for
 the second boundary value problem of mean curvature type equations.
\end{abstract}
\maketitle

\section{Introduction}
 The study of maximal and constant mean curvature spacelike hypersurfaces in Lorentzian spacetimes has attracted significant attention due to their importance in both physics and mathematics \cite{Bart1984,Cho1979,Mard1980}. In recent years, research on curvature equations arising from geometric problems, particularly for spacelike hypersurfaces in Minkowski space, has gained momentum \cite{TAE,Choi,LLJ,Bayard}. These equations play a central role in mathematical physics, including classical relativity and the positive mass conjecture, as demonstrated in the foundational work of Yau and Schoen \cite{yau-sch 1979}.

 A fundamental contribution to this field was made by Bartnik and Simon \cite{Bart1982}. Specifically, they considered the following Dirichlet boundary value problem:
 \begin{equation}\label{e1}
 	\left\{
 	\begin{array}{rlll}
 		\mathrm{div}\left(\dfrac{Du}{\sqrt{1-|Du|^2}}\right) &= H(x, u),   &\mathrm{in} & \Omega, \\
 		u &= \varphi,  &\mathrm{on} & \partial \Omega,
 	\end{array}
 	\right.
 \end{equation}
 where \( H: \Omega \times \mathbb{R} \to \mathbb{R} \) and \( \varphi \) are given bounded functions, and \( \Omega \subset \mathbb{R}^{n} \) is a bounded domain. They established necessary and sufficient conditions for the existence of smooth, strictly spacelike solutions. Bayard \cite{Bay2003} also solved  the Dirichlet problem for  prescribed scalar curvature equation in ambiant dimension 4. Bereanu, Jebelean, and Torres \cite{Bere2013} investigated a class of Dirichlet problems of the form
 \begin{equation}\label{e2}
 	\left\{
 	\begin{array}{rlll}
 		\mathrm{div}\left(\dfrac{Du}{\sqrt{1-|Du|^2}}\right) + \lambda \left[\mu(|x|) u^q\right] &= 0, & \mathrm{in} & B_R(0), \\
 		u &= 0, & \mathrm{on} & \partial B_R(0),
 	\end{array}
 	\right.
 \end{equation}
 where \(\lambda > 0\) is a parameter, \(q > 1\), \(R > 0\), and \(\mu: [0, \infty) \to \mathbb{R}\) is a continuous and strictly positive function. They established the existence of a critical value \(\Lambda > 0\) such that the problem admits at least one positive solution when \(\lambda > \Lambda\), while no radial positive solution exists for \(\lambda \in (0, \Lambda)\). There are also numerous significant results concerning the uniqueness and classification of constant mean curvature (CMC) hypersurfaces. For umbilical CMC spacelike hypersurfaces equipped with a closed conformal timelike vector field \(X\), Montiel \cite{Mon 1999} established several uniqueness theorems. Building on these results, Aquino and Lima demonstrated in \cite{Aquino1} that any complete spacelike hypersurface immersed with CMC or bounded mean curvature must be totally umbilical. Furthermore, Bonsante, Seppi, and Smillie \cite{Bon2023} proved that any regular domain in Minkowski space is uniquely foliated by spacelike CMC hypersurfaces.
 
 A celebrated result by Cheng and Yau \cite{Cheng1976} states that the only entire solution to the vanishing mean curvature equation in Minkowski space is a linear function, and this holds in all dimensions. Xin \cite{Xin0} extended this by showing that any complete spacelike hypersurface with constant mean curvature in Minkowski space, whose Gauss map \(\gamma: M \to H^n(-1)\) is bounded, must necessarily be a linear subspace. In the context of three-dimensional Lorentz–Minkowski space, Rafael \cite{Raf2007} proved that rotational symmetric surfaces are the only solutions for CMC spacelike surfaces with a circular boundary.
 For periodic maximal surfaces, Fernández and López \cite{Fern2007} provided a classification of complete flat Lorentzian 3-manifolds that admit entire maximal surfaces of finite type. Recently, a one-to-one correspondence between solutions of infinite boundary value problems for minimal surfaces and those of lightlike line boundary value problems for maximal surfaces in Lorentz-Minkowski spacetime was established in \cite{Akm2024}.  
 
    For other types of boundary value problems, Mawhin and Torres \cite{Maw2016} investigated the Neumann boundary condition for the prescribed mean curvature surface problem in \(B_R(0)\) within the framework of FLRW spacetime. The spacetime is characterized by the metric \(ds^2 = -dt^2 + f(t)dx^2\), where \(f(t) \in C^1[a, b]\). Under suitable assumptions on \(f\), they proved the existence of at least one radially symmetric solution for any prescribed mean curvature \(H\).
    Further contributions to this field include the work of Urbas, who established the existence of Weingarten hypersurfaces with prescribed gradient images in \cite{J2}. In \cite{Ju3}, Urbas constructed a smooth pseudoconvex pair \((D_1, D_2)\) of domains in \(\mathbb{R}^2\) with equal areas, showing that no globally smooth minimal Lagrangian diffeomorphism exists from \(\overline{D_1}\) onto \(\overline{D_2}\).
   
 The study of fully nonlinear partial differential equations with second boundary value conditions has a long and rich history. Urbas \cite{Ju2,Ju1} investigated the existence of globally smooth solutions to Monge-Ampère-type equations and a class of Hessian equations by transforming the second boundary condition into an oblique derivative problem. Nessi and Gregory \cite{Nessi} established the existence of globally smooth classical solutions for a new class of modified Hessian equations, closely related to the optimal transportation equation, which satisfies the second boundary value problem. Jiang and Trudinger studied the global regularity of oblique boundary value problems for augmented Hessian equations for a class of general operators and developed a comprehensive theory for a priori derivative estimates up to the second order. In \cite{Jiang1}, they also obtained classical global \(C^3\) solutions to second boundary value problems for generated prescribed Jacobian equations under suitable assumptions on the target domain and mapping function. Shibing Chen et al. \cite{Chen1} established global \(C^{2,\alpha}\) and \(W^{2,p}\) regularity for the Monge-Ampère equation subject to the second boundary condition. Later, Savin et al. \cite{Savin} studied the global regularity of \(W^{2,1+\epsilon}\) estimates for the Monge-Ampère equation under the second boundary condition and demonstrated the optimality of the exponent \(1+\epsilon\).
 
 For the Lagrange mean curvature equation with the second boundary value condition:
 \begin{equation}\label{Ftauc}
 	\left\{
 	\begin{array}{rl}
 		\sum\limits_{i=1}^{n} \arctan \left(\lambda_{i}(D^2u)\right) &= c, \quad x \in \Omega, \\
 		Du(\Omega) &= \tilde{\Omega}.
 	\end{array}
 	\right.
 \end{equation}
 Brendle and Warren \cite{Brendle2008ABV} proved the existence and uniqueness of solutions to \eqref{Ftauc} using elliptic methods, while Huang \cite{Huang2014OnTS} obtained the existence of solutions by considering the second boundary value problem for the Lagrangian mean curvature flow. Wang, Huang, and Bao studied the Lagrangian mean curvature-type equation:
 \begin{equation}\label{preeq}
 	\left\{
 	\begin{array}{rl}
 		\sum\limits_{i=1}^{n} \arctan \left(\lambda_{i}(D^2u)\right) &= \kappa \cdot x+ c, \quad x \in \Omega, \\
 		Du(\Omega) &= \tilde{\Omega},
 	\end{array}
 	\right.
 \end{equation}
where \(\kappa\) is a constant vector, and the Lagrangian graph \((x, Du(x))\) has a prescribed constant mean curvature vector \(H = (0, \kappa)^{\perp}\) in Euclidean space. Here, \(Du\) is a diffeomorphism between two uniformly convex bounded domains. In \cite{Wang2023}, the authors proved the existence and uniqueness of smooth uniformly convex solutions when \(\kappa\) is sufficiently small, and they extended these results to the corresponding parabolic setting in \cite{Wang2024}. Recently, Jiang and Bao \cite{Jiangbao2025} studied the translator solution case, where the right-hand side of \eqref{preeq} takes the form \(\iota \cdot Du + \kappa \cdot x + c\) for a constant vector \(\iota\). They established analogous results when both \(\iota\) and \(\kappa\) are sufficiently small. In a series of works \cite{HR1}--\cite{HW}, Huang et al. investigated the second boundary value problem for various nonlinear equations using elliptic or parabolic methods and explored their geometric applications.

Motivated by the above works, 
We consider convex spacelike hypersurfaces with mean curvature in Minkowski space $\mathbb{R}^{n,1}$ with the Lorentzian metric
\begin{equation}\label{e1.1.1}
	ds^2=\sum^n_{i=1}dx_i^2-dx_{n+1}^2.
\end{equation}
Any such hypersurface can be written locally as a graph of a function $x_{n+1} = u(x)$, $x\in \mathbb{R}^n$, satisfying the spacelike condition
\begin{align}
    |Du|<1.
\end{align}

In this work, we investigate the existence and uniqueness of uniformly convex solutions to the mean curvature type equation
\begin{equation}\label{e1.1.3}
	\mathrm{div}\left(\frac{Du}{\sqrt{1-|Du|^2}}\right) = f(x, Du) + c, \quad x \in \Omega,
\end{equation}
coupled with the second boundary value condition
\begin{equation}\label{e1.1.4}
	Du(\Omega) = \tilde{\Omega},
\end{equation}
where \(f \in C^{2+\alpha}(\overline{\Omega}\times\overline{\tilde{\Omega}})\), 
and \(\Omega\) and \(\tilde{\Omega}\) are uniformly convex bounded domains with smooth boundaries in \(\mathbb{R}^n\). The boundary condition \eqref{e1.1.4} arises naturally for mean curvature type equation \eqref{e1.1.3}, as these equations are elliptic precisely on locally uniformly convex solutions. In this case, the gradient map \(Du\) becomes a diffeomorphism from \(\Omega\) onto its image \(\overline{Du(\Omega)} \subset B_1(0)\), where \(B_1(0)\) denotes the unit ball in \(\mathbb{R}^n\) equipped with the Klein model of hyperbolic geometry, represented by \(\{(x, 1) \in \mathbb{R}^{n,1} \mid |x| < 1\}\). 
Denote
\[\mathscr{A}:=\left\{f(x,p)\in C^{2+\alpha}(\overline{\Omega}\times\overline{\tilde{\Omega}}): f ~is~ concave ~in~x\right\}.\]

 Our main Theorems are stated as follows.
\begin{theorem}\label{t1.1}
 Suppose that \(\Omega\) and \(\tilde{\Omega}\) are uniformly convex bounded domains with smooth boundaries in \(\mathbb{R}^n\), and \(\tilde{\Omega} \subset\subset B_1(0)\). If \(f \in \mathscr{A}\) and the oscillation of \(f\), together with \(|D_x f|\) and \(|D_{xp} f|\), are sufficiently small, then there exists a uniformly convex solution \(u \in C^{4,\alpha}(\overline{\Omega})\) and a constant \(c\) satisfying \eqref{e1.1.3} and \eqref{e1.1.4}. The constant \(c\) depends only on \(n\), \(\Omega\), \(\tilde{\Omega}\), and \(\|f\|_{C^{2+\alpha}(\overline{\Omega} \times \overline{\tilde{\Omega}})}\). Furthermore, if \(f\) is smooth, then \(u \in C^{\infty}(\overline{\Omega})\). Here, \(u\) is unique up to a constant.
\end{theorem}
\begin{remark}
	It is well known that certain necessary conditions must be satisfied for the solvability of the second boundary value problem for elliptic partial differential equations \cite{urbas1,urbas2,ma}. In our context, the constant \(c\) serves to adjust the range of the function on the right-hand side, ensuring the solvability of the problem, see also Lemma \ref{L2.1}. This approach provides an alternative adjustment method distinct from the one employed in \cite{HOWY}.
\end{remark}


We go on considering the mean curvature type equation  in $\mathbb{R}^{n+1}$:
\begin{equation}\label{e12a}
\mathrm{div}\left(\frac{Du}{\sqrt{1+|Du|^2}}\right)=f(x,Du)+c,  \quad x\in \Omega,
\end{equation}
associated with the second boundary value problem
\begin{equation}\label{e12b}
Du(\Omega)=\tilde{\Omega}.
\end{equation}

Based on the same proof as Theorem \ref{t1.1}, an immediate consequence of  the above problem  is the following:
\begin{theorem}\label{t1.2}
Suppose that \(\Omega\) and \(\tilde{\Omega}\) are uniformly convex and bounded domains with smooth boundaries in \(\mathbb{R}^n\). If \(f \in \mathscr{A}\) and the oscillation of \(f\), together with \(|D_x f|\) and \(|D_{xp} f|\), are sufficiently small, then there exists a uniformly convex solution \(u \in C^{4,\alpha}(\overline{\Omega})\) and a constant \(c\) satisfying \eqref{e12a} and \eqref{e12b}. The constant \(c\) depends only on \(n\), \(\Omega\), \(\tilde{\Omega}\), and \(\|f\|_{C^{2+\alpha}(\overline{\Omega} \times \overline{\tilde{\Omega}})}\). Furthermore, if \(f\) is smooth, then \(u \in C^{\infty}(\overline{\Omega})\). Here, \(u\) is unique up to a constant.
\end{theorem}

The rest of this article is organized as follows. In section \ref{sec2}, we introduce some basic formulas and notations, and then present the structure condition for the mean curvatrue type equation. By the second boundary value condition \eqref{e1.1.4}, the gradient \( Du \) is bounded. Consequently, in section \ref{sec3}, we devote to carry out the strictly oblique estimate and then in section \ref{sec4} we obtain the  $C^2$ estimate according to the structure properties of the operators $G$ and $\tilde{G}$. We will prove the main theorem by the continuity method as same as Wang-Huang-Bao \cite{Wang2023} in section \ref{sec5}.

Throughout the following, Einstein’s convention of summation over repeated indices
will be adopted. We denote, for a smooth function $f(x,p)$,
\[f_{i}=\frac{\partial f}{\partial x_{i}},\quad f_{p_i}=\frac{\partial f}{\partial p_{i}},\quad f_{ij}=\frac{\partial^{2}f}{\partial x_{i}\partial x_{j}},\quad f_{p_ip_j}=\frac{\partial^{2}f}{\partial p_{i}\partial p_{j}},\quad f_{ijk}=\frac{\partial^{3}f}{\partial x_{i}\partial x_{j}\partial x_{k}},\ldots.\]

\section{Preliminaries}\label{sec2}
In this section, we will derive some basic formulas for the geometric quantities of spacelike hypersurfaces in Minkowski space $\mathbb{R}^{n,1}$.
We then give the  structure condition for the mean curvature type equation referring to \cite{Wang2023}.

We start with the definitions and notations of differential geometry for graphic hypersurface in Minkowski space $\mathbb{R}^{n,1}$,
the readers can see \cite{EH} and \cite{Ec} for a nice introduction.
A spacelike hypersurface $M\subset \mathbb{R}^{n,1}$ is a codimension one submanifold  with the Lorentzian metric \eqref{e1.1.1}
whose induced metric is Riemannian. Locally $M$ can be written as a graph
\begin{align}
    \mathcal{M}_u = \{X = (x, u(x))|x \in \mathbb{R}^n\}
\end{align}
satisfying the spacelike condition $(\ref{e1.1.4})$ and $\mathcal{M}_u$ is a uniformly convex hypersurface. Let $\mathcal{N}$ be the timelike unit normal vector to $\mathcal{M}_u$ at $x$, one can defiine the Gauss map $\mathscr{G}$:
\begin{align*}
	 \mathcal{M}_u\longrightarrow S_{1,+}^n: x\longrightarrow\mathcal{N}(x)
\end{align*}
where $S_{1,+}^n:=\left\{\left(x_1,\cdots,x_n,x_{n+1}\right)|x_1^2+\cdots+x_n^2-x_{n+1}^2=-1,x_{n+1}\geq0\right\}$ and 
$$\mathcal{N}(x):=\dfrac{(-Du,1)}{\sqrt{1-|Du|^2}}.$$
If we take the hyperplane
$$\mathcal{P}:=\left\{X=(x_1,\cdots,x_n,x_{n+1})|x_{n+1}=1\right\}$$
and consider the projection of $S_{1,+}^n$ from the origin into $\mathcal{P}$. Then $S_{1,+}^n$ is mapped a one-to-one fashion onto the ball $B=\left\{(y_1,\cdots,y_n)\in \mathbb{R}^n|y_1^2+ \cdots y_n^2\leq1\right\}$.
Then map $\mathscr{P}$ is given by 
$$S_{1,+}^n\longrightarrow B: (x_1,\cdots,x_n,x_{n+1})\to(y_1,\cdots,y_n)$$
where $x_{n+1}=\sqrt{1+x_1^2+ \cdots x_n^2}$ and $ y_i=-\frac{x_i}{x_{n+1}}$. Thus it yields that
$$\mathscr{P}\circ \mathscr{G}:\left(x,u(x)\right)\in \mathcal{M}_u\longrightarrow Du(x)\in B.$$
Then the gradient map $Du$ is equivalent to the Gauss map(see Lemma 4.5 in \cite{Choi}).

It is easy to see that the induced metric and second
fundamental form of $M $ are given by
\begin{align}
g_{ij}=\delta_{ij}-D_{i}uD_{j}u,\quad 1\leq i,j\leq n.
\end{align}

While the inverse of the induced metric and second fundamental form of $M$ are given respectively by
\begin{align}
g^{ij}=\delta_{ij}+\frac{D_{i}uD_{j}u}{1-|Du|^{2}},\quad 1\leq i,j\leq n,
\end{align}
and
\begin{align}
h_{ij}=\frac{D_{ij}u}{\sqrt{1-|Du|^{2}}},\quad 1\leq i,j\leq n.
\end{align}
where $|Du|=\sqrt{\sum_{i=1}^{n}|D_{i}u|^{2}}$. Specially, the mean curvature of $M$ is written as
\begin{equation}\label{e1.11}
H=\sum_{1\leq i\leq n}\kappa_{i}.
\end{equation}
where $\kappa_{i},i=1,\cdots,n$  are the principal curvatures of $M\subset \mathbb{R}^{n,1}$.

We aim to construct convex spacelike mean curvature  type hypersurfaces with prescribed Gauss
image and some types of mean curvature over any strictly convex domains by solving problem \eqref{e1.1.3} and \eqref{e1.1.4}.

It follows from \cite{LLJ}  that  we can state various geometric quantities associated with the graph of  $u\in C^{2}(\Omega)$.
In the coordinate systems, Latin indices range from 1 to $n$ and indicate quantities in the graph.
We adopt the Einstein's convention of summation over repeated indices in the following.\\
Let
\begin{equation}
        F(\kappa_{1},\cdots, \kappa_{n}):=\sum_{i=1}^n\kappa_{i},
\end{equation}
be a smooth function on the positive cone
\begin{align*}
    \Gamma^+_n:=\left\{(\kappa_1,\cdots,\kappa_n)\in \mathbb{R}^n:\kappa_i>0,\ i=1,\cdots,n\right\}.
\end{align*}

  The $F$ is a smooth symmetric function defined on $\Gamma^+_n$. Acorrding to  \cite{ou}, the $F$ satisfies
\begin{equation}\label{e2.2.10}
\frac{\partial F}{\partial \kappa_i}>0,\ \ 1\leq i\leq n\ \  \text{on}\ \  \Gamma^+_n,
\end{equation}
\begin{equation}
\sum_{i=1}^n\frac{\partial F}{\partial \kappa_i}=n \ \   \text{on}\ \  \Gamma^+_n,
\end{equation}
and
  \begin{equation}
	\sum_{i=1}^n\frac{\partial F}{\partial \kappa_i}\kappa_i= F,
\end{equation}
and
\begin{equation}\label{e2.2.12}
\left(\frac{\partial^2 F}{\partial \kappa_i\partial \kappa_j}\right)\leq 0\ \  \text{on}\ \  \Gamma^+_n.
\end{equation}

The following two lemmas give the structural conditions of $F$. 
\begin{lemma}\label{L2.1}
    Assume that $\Omega, \tilde{\Omega}$ are bounded, uniformly convex domains with smooth boundary in $\mathbb{R}^n$ and  $\tilde{\Omega}\subset\subset B_1(0)$. Suppose the oscillation of $f$ satisfing
   \begin{equation}\label{oscf}
   	osc(f):=\max_{x,y\in \overline{\Omega},p,q\in \overline{\tilde{\Omega}}}|f(x,p)-f(y,q)|\leq \frac{n}{2}\left( \frac{|\tilde{\Omega}|}{|\Omega|}\right)^{\frac{1}{n}}.
   \end{equation}
    where $|\Omega|$ is the volume of $\Omega$.
     If the strictly convex solution to \eqref{e1.1.3} and \eqref{e1.1.4} exists, then there exist positive constants $\Lambda_1$ and $\Lambda_2$, depending only on $n,\Omega, \tilde{\Omega}$ and $f$, such that there holds
\begin{align}\label{e2.2.8}
    \Lambda_1 \leq F(\kappa) \leq \Lambda_2.
\end{align}
\end{lemma}
\begin{proof}
Since $Du(\Omega)=\tilde{\Omega}$, $\tilde{\Omega} \subset\subset {B}_1(0)$,  we have
\begin{align*}
	|\tilde{\Omega}|&=\int_{\tilde{\Omega}}dy=\int_{\Omega}\det D^2udx\\&=\int_{\Omega}\left(1-|Du|^2\right)^{\frac{n+2}{2}}\frac{\det D^2u}{\left(1-|Du|^2\right)^{\frac{n+2}{2}}}dx.
\end{align*}
By noting that
\begin{align*}
	\kappa_1\cdot\kappa_2\cdots\kappa_n=\frac{\det D^2u}{\left(1-|Du|^2\right)^{\frac{n+2}{2}}},
\end{align*}
then we can get
\begin{align*}
 \int_{\Omega} \left(\frac{\kappa_1+\cdots+\kappa_n}{n}\right)^n  dx&\geq \int_{\Omega}\frac{\det D^2u}{\left(1-|Du|^2\right)^{\frac{n+2}{2}}}dx\\
 &\geq\int_{\Omega}\det D^2udx=|\tilde{\Omega}|
\end{align*}
Then we can find $\bar{x}\in \Omega$ such that 
$$F(\kappa)|_{\bar{x}}\geq n\left( \frac{|\tilde{\Omega}|}{|\Omega|}\right)^{\frac{1}{n}}.$$
Since the functions \(f\) and \(F\) share the same oscillation, it follows \eqref{oscf} that
$$F(\kappa)\geq  \frac{n}{2}\left( \frac{|\tilde{\Omega}|}{|\Omega|}\right)^{\frac{1}{n}}.$$
Integrating over 
on the both sides of \eqref{e1.1.3}, we can obtain
\begin{align*}
	\int_{\Omega}\mathrm{div}\left(\frac{Du}{\sqrt{1-|Du|^{2}}}\right)dx-\int_{\Omega}fdx=c|\Omega|,
\end{align*}
By using the
divergence theorem, we see that
\begin{align*}
c=\frac{1}{|\Omega|}\int_{\partial\Omega}\frac{Du\cdot\nu}{\sqrt{1-|Du|^{2}}}ds-\frac{1}{|\Omega|}\int_{\Omega}fdx,
\end{align*}
where $\nu=(\nu_1,\nu_2,\cdots,\nu_n)$ is  the unit outward normal vector of $\partial\Omega$, and then
\begin{align*}
	c&\leq \frac{1}{|\Omega|}\int_{\partial\Omega}\frac{|Du|}{\sqrt{1-|Du|^{2}}}ds+\max_{\overline{\Omega}\times\overline{\tilde{\Omega}}}|f(x,p)|.
\end{align*}
Therefore,
$$F(\kappa)\leq \frac{|\partial\Omega|}{|\Omega|}\operatorname*{max}_{y\in\partial\tilde{\Omega}}\frac{|y|}{\sqrt{1-|y|^{2}}}+2\max_{\overline{\Omega}\times\overline{\tilde{\Omega}}}|f(x,p)|.$$
Thus the proof of \eqref{e2.2.8} is
completed.
\end{proof}
\begin{lemma}\label{e3.2.2} Assume that $\Omega$, $\tilde{\Omega}$ are bounded, uniformly convex domains with smooth boundary in $\mathbb{R}^{n}$ and   $\tilde{\Omega}\subset\subset B_1(0)$. Suppose $f$ satisfy \eqref{oscf} and  $u$  is the strictly convex solution to $(\ref{e1.1.3})$ and $(\ref{e1.1.4})$. Then there exist positive constants $\Lambda_3$ and $\Lambda_4$,  such that there holds
  \begin{align}
 \label{e3.1}
 \Lambda_3\leq\sum^{n}_{i=1}\frac{\partial F}{\partial \kappa_{i}}\leq \Lambda_4
 \end{align}
and
\begin{equation}\label{e3.2}
\Lambda_3\leq\sum^{n}_{i=1}\frac{\partial F}{\partial \kappa_{i}}\kappa^{2}_{i}\leq \Lambda_4.
\end{equation}
\end{lemma}
\begin{proof}
  We observe that the operator
 $$F=\sum_{i=1}^n\kappa_{i}$$ and $$\frac{\partial F}{\partial \kappa_{i}}=1.$$
Then $$\sum^{n}_{i=1}\frac{\partial F}{\partial \kappa_{i}}=n,$$
with
$$\sum^{n}_{i=1}\frac{\partial F}{\partial \kappa_{i}}\kappa^{2}_{i}=\sum^{n}_{i=1}\kappa^{2}_{i},$$
and using Cauchy inequality, we obtain
$$ \frac{1}{n}(\kappa_1+\cdots+\kappa_n)^2\leq \kappa_1^2+\cdots+\kappa_n^2\leq (\kappa_1+\cdots+\kappa_n)^2 ,$$
combining with \eqref{e2.2.8} in Lemma \ref{L2.1}, we know $F$ is bounded and we obtain the
desired results.
\end{proof}

The principal curvatures of $M\subset \mathbb{R}^{n,1}$ are the eigenvalues of the second fundamental form
$h_{ij}$ relative to $g_{ij}$, i.e., the eigenvalues of the mixed tensor $h^{j}_{i}\equiv h_{ik}g^{kj}$.
By \cite{LLJ} we remark that they are the eigenvalues of the symmetric matrix
\begin{equation}\label{e2.2.13}
	a_{ij}=\frac{1}{v}b^{ik}D_{kl}ub^{lj},
\end{equation}
where $v=\sqrt{1-|Du|^{2}}$ and $b^{ij}$ is the positive square root of $g^{ij}$ taking the form
\begin{align*}
	b^{ij}=\delta_{ij}+\frac{D_{i}uD_{j}u}{v(1+v)}.
\end{align*}
The inverse of $b^{ij}$ is
\begin{align*}
	b_{ij}=\delta_{ij} - \frac{D_{i}uD_{j}u}{1+v},
\end{align*}
which is the square root of $g_{ij}$.
And then
\begin{align*}
	D_{ij}u=vb_{ik}a_{kl}b_{lj}.
\end{align*}
By some calculation, it yields
\begin{align*}
	a_{ij}=\frac{1}{v}\left (D_{ij}u - \frac{D_iuD_luD_{jl}u}{v(1+v)} - \frac{D_juD_luD_{il}u}{v(1+v)} + \frac{D_iuD_kuD_luD_juD_{kl}u}{v^2(1+v)^2}\right).
\end{align*}

Denote $\mathcal{A}=[a_{ij}]$ and $F[\mathcal{A}]=\sum_{i=1}^n\kappa_{i}$. Then the properties of the operator $F$ are reflected in (\ref{e2.2.8})-(\ref{e3.2}).
It follows from (\ref{e2.2.10}) that
\begin{equation*}
	F_{ij}[\mathcal{A}]\xi_{i}\xi_{j}>0 \quad \mathrm{for}\quad \mathrm{all}\quad \xi\in \mathbb{R}^{n}-\{0\},
\end{equation*}
where
\begin{equation*}
	F_{ij}[\mathcal{A}]=\frac{\partial F[\mathcal{A}]}{\partial a_{ij}}.
\end{equation*}

From \cite{JS} we see that $[F_{ij}]$ is diagonal if $\mathcal{\mathcal{A}}$ is diagonal, and in this case
\begin{equation*}
	[F_{ij}]=\mathrm{diag}(\frac{\partial F}{\partial \kappa_{1}},\cdots, \frac{\partial F}{\partial \kappa_{n}})=\mathrm{diag}(1,\cdots,1).
\end{equation*}

If $u$ is convex, by (\ref{e2.2.13}) we deduce that the  eigenvalues of the  matrix $[a_{ij}]$ must be in $\overline{\Gamma^+_n}$.
Then (\ref{e2.2.12}) implies  that
\begin{align*}
	F_{ij,kl}[\mathcal{A}]\eta_{ij}\eta_{kl}\leq 0,
\end{align*}
for any real symmetric matrix $[\eta_{ij}]$, where
\begin{equation*}
	F_{ij,kl}[\mathcal{A}]=\frac{\partial^{2}F[\mathcal{A}]}{\partial a_{ij}\partial a_{kl}}.
\end{equation*}

According to the equation (\ref{e1.1.3}), we consider the nonlinear differential operators of the type
\begin{equation*}\label{e2.10}
G(Du,D^{2}u)=f(x,Du)+c.
\end{equation*}

As in \cite{J2}, by differentiating this equation once, we have
\begin{equation*}
G_{ij}D_{ijk}u+G_{p_i}D_{ik}u=f_k+f_{p_i}u_{ik},
\end{equation*}
where we use the notation
\begin{equation*}
G_{ij}=\frac{\partial G}{\partial r_{ij}} \quad \text{and} \quad G_{p_i}=\frac{\partial G}{\partial p_{i}},
\end{equation*}
with $r$ and $p$  representing for the second derivatives and gradient variables respectively.
So as to prove the strict obliqueness estimate for the problem (\ref{e1.1.3})-(\ref{e1.1.4}), we need to recall some expressions from \cite{J2} for the derivatives of $G$.  A simple calculation yields
\begin{equation}\label{e2.2.14}
G_{ij}=F_{kl}\frac{\partial a_{kl}}{\partial r_{ij}}=\frac{1}{v}b^{ik}F_{kl}b^{lj}
\end{equation}
and
\begin{equation*}\label{e2.12}
G_{p_i}=F_{kl}\frac{\partial a_{kl}}{\partial p_i}=\frac{u_i}{v^2}F_{kl}a_{kl}+\frac{2}{v}F_{kl}a_{ml}b^{ik}u_m.
\end{equation*}

We observe that $\mathcal{T}_{G}=\sum^{n}_{i=1}G_{ii}$ is the trace of a product of three matrices by (\ref{e2.2.14}), so it is invariant under orthogonal transformations. Hence, to compute $\mathcal{T}_{G}$, we may assume for now that $[a_{ij}]$ is diagonal. By virtue of (\ref{e1.1.4}) and $\tilde{\Omega}\subset\subset B_{1}(0)$, we obtain that $Du$ and $\frac{1}{v}$  are bounded. Then eigenvalues of $[b^{ij}]$ are bounded between two controlled positive constants.
We can observe that
\begin{align*}
\mathcal{T}_{G}=\sum_{i=1}^nG_{ii}=\frac{1}{v}b^{ik}F_{kl}b^{li},
\end{align*}
it follows that there exist positive constants $\sigma_{1}$,
$\sigma_{2}$ depending only on the least upper bound of $|Du|$ in the set $\Omega$, such that
\begin{equation}\label{e2.14}
\sigma_{1}\mathcal{T}\leq\mathcal{T}_{G}\leq\sigma_{2}\mathcal{T},
\end{equation}
where $\mathcal{T}=\sum^{n}_{i=1}F_{ii}$.
By the concavity of $F$ and the positive definiteness of $[F_{ij}a_{ij}]$ imply that$[F_{ij}a_{ij}]$ is controlled by $F$, i.e.,
\begin{equation*}\label{e2.15}
0<F_{ij}a_{ij}=\sum^{n}_{i=1}\frac{\partial F}{\partial \kappa_{i}}\kappa_{i}= F(\kappa_{1},\ldots,\kappa_{n}).
\end{equation*}
Thus $G_{p_i}$ is bounded.

Next, we will use the Lengendre transformation of $u$ which is the convex function $\tilde{u}$ on $\Tilde{\Omega}=Du(\Omega)$ define by \\
\begin{align*}
    \tilde{u}(y)=x\cdot Du(x)-u(x),
\end{align*}
and
\begin{align*}
    y=Du(x).
\end{align*}
It follows that
\begin{align*}
    \frac{\partial \tilde{u}}{\partial y_{i}}=x_{i},\quad\frac{\partial ^{2}\tilde{u}}{\partial y_{i}\partial y_{j}}=u^{ij}(x),
\end{align*}
where $[u^{ij}]=[D^2 u]^{-1}$. Then $\tilde{u}$ satisfies
\begin{align}\label{transfoemation G}
    \tilde{G}(y,D^2\tilde{u})=-G(y,[D^2\tilde{u}]^{-1})=-f(D\tilde{u},y)-c \qquad \text{in} \quad \tilde{\Omega},
\end{align}
with the boundary condition
\begin{align*}
    D\tilde{u}(\tilde{\Omega})=\Omega.
\end{align*}
Moreover, $(\ref{transfoemation G})$ can be written as
\begin{align*}
    \tilde{F}[\tilde{a}_{ij}]=-f(D\tilde{u},y),
\end{align*}
where
\begin{align*}
    \tilde{F}[\tilde{a}_{ij}]=\tilde{F}(\eta_1,\eta_2,\cdots,\eta_n)=-F(\eta_1^{-1},\eta_2^{-1},\cdots,\eta_n^{-1}).
\end{align*}

Here $\tilde{F}$ satisfies the structural conditions of Lemma \ref{e3.2.2}, $\eta_1,\eta_2,\cdots,\eta_n$ are the eigenvalues of the matrix [$\Tilde{a}_{ij}$] and
\begin{align*}
    \tilde{v}=\sqrt{1-|y|^2},
\end{align*}
\begin{align*}
    [\tilde{a}_{ij}]=\sqrt{1-|y|^2}\tilde{b}_{ik}D_{kl}\tilde{u}\tilde{b}_{lj},
\end{align*}
\begin{align*}
    \tilde{b}_{ij}=\delta_{ij}-\frac{y_iy_j}{1+\tilde{v}},
\end{align*}
where $[\tilde{b}^{ij}]$ is denoted the inverse matrix of $[\tilde{b}_{ij}]$ , it is given by
\begin{align*}
    \tilde{b}^{ij}=\delta_{ij}+\frac{y_iy_j}{\tilde{v}(1+\tilde{v})}.
\end{align*}

Since $y\in\tilde{\Omega}$, the eigenvalues of $[\tilde{b}_{ij}]$ and $[\tilde{b}^{ij}]$ are bounded between two controlled positive constants $\sigma_3,\sigma_4$. Consequently, we have
\begin{align}\label{e2.2.19}
\sigma_3\tilde{\mathcal{T}}\leq\tilde{\mathcal{{T}}}_{\tilde{G}}\leq \sigma_4\tilde{\mathcal{T}},
\end{align}
where $\tilde{\mathcal{T}}=\sum_{i=1}^n\tilde{F}_{ii}$, $\Tilde{\mathcal{T}}_{\tilde{G}}=\sum_{i=1}^n\Tilde{G}_{ii}$, we can conclude that
\begin{align*}
\tilde{G}_{ij}D_{ijk}\tilde{u}+\tilde{G}_{y_k}=-f_i\tilde{u}_{ik}-f_{p_k},
\end{align*}
where
\begin{align*}
    \tilde{G}_{ij}=\sqrt{1-|y|^2}\tilde{b}_{ik}\tilde{F}_{kl}\tilde{b}_{lj},
\end{align*}
and
\begin{align*}    
	\tilde{G}_{y_i}&=\tilde{F_{kl}}\frac{\partial}{\partial y_i}\left(\sqrt{1-|y|^2}\tilde{b}_{kp}\tilde{b}_{ql}\right)\tilde{u}_{pq}\\
&=-\frac{y_i}{1-|y|^2}\tilde{F}_{kl}\tilde{a}_{kl}+2\tilde{F}_{kl}\tilde{a}_{lm}\frac{\partial}{\partial y_i}(\tilde{b}_{kp})\cdot \tilde{b}^{pm}\\
 &=-\frac{y_i}{1-|y|^2}\tilde{F}_{kl}\tilde{a}_{kl}+\tilde{F}_{kl}\tilde{a}_{lm}\left(-\frac{\delta_{ik}y_p}{1+\tilde{v}}-\frac{y_{k}\delta_{ip}}{1+\tilde{v}}-\frac{y_ky_i}{(1+\tilde{v})^2\tilde{v}}\right)\cdot \left(\delta_{pm}+\frac{y_py_m}{\tilde{v}(1+\tilde{v})}\right)\\
 &=-\frac{y_i}{1-|y|^2}\tilde{F}_{kl}\tilde{a}_{kl}+\tilde{F}_{kl}\tilde{a}_{lm}c_{ikmp},
\end{align*}
$c_{ikmp}$ depending on $y,y_i,y_k,y_m,y_p$ and $c$ is a constant. We can similarly show that  $\tilde{G}_{y_i}$ is bounded. Therefore, there holds
  \begin{align}
 |\tilde{G}_{y_i}|\leq \Lambda_5,
  \end{align}
 where $\Lambda_5$ is a uniformly positive constant. 
 
Next, we can give the structural conditions for the operator $G$ and $\tilde{G}$.
 \begin{Corollary}\label{c3.3}
 	Assume that $\Omega$, $\tilde{\Omega}$ are bounded, uniformly convex domains with smooth boundary in $\mathbb{R}^{n}$ and  $\tilde{\Omega}\subset\subset B_{1}(0)$. Suppose that Suppose $f$ satisfy \eqref{oscf} and  $u \in C^2(\overline{\Omega})$ is the strictly convex solution to \eqref{e1.1.3} and  \eqref{e1.1.4}.
 	There   exists  uniformly  positive constants  $\Lambda_6,\Lambda_7$,  depending  on $n,\Omega,\tilde{\Omega}$ and $f$,  such that there holds
 	\begin{equation}\label{e3.4}
 		\Lambda_6\leq \sum^{n}_{i=1}\frac{\partial G}{\partial \lambda_{i}}\leq \Lambda_7,
 	\end{equation}
 	\begin{equation}\label{e3.5}
 		\Lambda_6\leq \sum^{n}_{i=1}\frac{\partial G}{\partial \lambda_{i}}\lambda^{2}_{i}\leq \Lambda_7,
 	\end{equation}
 	where $\lambda_1,\cdots, \lambda_n$ are the eigenvalues of Hessian matrix $D^2 u$ at $x \in \Omega$.
 \end{Corollary}
 Gathering the results obtained above we arrive at the following structural conditions for the operator $\tilde{G}$.
 \begin{Corollary}\label{c3.4}
 	Assume that $\Omega$, $\tilde{\Omega}$ are bounded, uniformly convex domains with smooth boundary in $\mathbb{R}^{n}$
 	and  $\tilde{\Omega}\subset\subset B_{1}(0)$. Suppose that Suppose $f$ satisfy \eqref{oscf} and  $u \in C^2(\overline{\Omega})$ is a strictly convex solution of \eqref{e1.1.3} and  \eqref{e1.1.4}.
 	Then there exists uniformly  positive constants  $\Lambda_8,\Lambda_9$ which depending on $n,\Omega,\tilde{\Omega}$ and $f$, such that there holds
 	\begin{equation}\label{e3..3.3}
 		\Lambda_8\leq \sum^{n}_{i=1}\frac{\partial \tilde{G}}{\partial \mu_{i}}\leq \Lambda_9,
 	\end{equation}
 	\begin{equation}\label{e3..3.4}
 		\Lambda_8\leq \sum^{n}_{i=1}\frac{\partial \tilde{G}}{\partial \mu_{i}}\mu^{2}_{i}\leq \Lambda_9,
 	\end{equation}
 	where $\mu_1,\cdots, \mu_n$ are the eigenvalues of the Hessian matrix $D^2 \tilde{u}$ at $x \in \Omega$.
 \end{Corollary}

\section{The strict obliqueness estimate}\label{sec3}
In this section, 
We will carry out the strict obliqueness estimates.

From \cite{Ju1}, we give the following
\begin{deff}
A smooth function $h:\mathbb{R}^n\rightarrow\mathbb{R}$ is called the defining function of $\tilde{\Omega}$,
 if
$$\tilde{\Omega}=\{p\in\mathbb{R}^{n} : h(p)>0\},\quad |Dh|_{{\partial\tilde{\Omega}}}=1,$$
and there exists $\theta>0$ such that for any $p=(p_{1},\cdots, p_{n})\in \tilde{\Omega}$ and $\xi=(\xi_{1}, \cdots, \xi_{n})\in \mathbb{R}^{n}$,
$$\frac{\partial^{2}h}{\partial p_{i}\partial p_{j}}\xi_{i}\xi_{j}\leq -\theta|\xi|^{2}.$$
\end{deff}
Therefore, the diffeomorphism condition $Du(\Omega) =\tilde\Omega$ in (\ref{e1.1.4}) is equivalent to
\begin{align}
    h(Du)=0,\quad x\in\partial\Omega.
\end{align}

And then the mean curvature equation (\ref{e1.1.3})-(\ref{e1.1.4}) is equivalent to the following elliptic problem
\begin{equation}\label{e3.3.4}
\begin{cases}
G(Du,D^2u)=f(x,Du)+c, \quad x\in \Omega,\\
\qquad\;\;  h(Du)=0,\quad x\in \partial\Omega.
 \end{cases}
\end{equation}
\begin{lemma}$($See Lemma 3.4 in \cite{RS}$)$\label{l3.30}
Assume that  $[A_{ij}]$ is semi-positive real symmetric matrix and $[B_{ij}]$, $[C_{ij}]$ are two real symmetric matrices. Then
$$2A_{ij}B_{jk}C_{ki}\leq A_{ij}B_{ik}B_{jk}+A_{ij}C_{ik}C_{jk}.$$
\end{lemma}
For the convenience, we denote $\beta=(\beta^{1}, \cdots, \beta^{n})$ with $\beta^{i}:=h_{p_{i}}(Du)$, and $\nu=(\nu_{1},\cdots,\nu_{n})$ as the unit inward normal vector at $x\in\partial\Omega$. The expression of the inner product is
\begin{equation*}
	\langle\beta, \nu\rangle=\beta^{i}\nu_{i}.
\end{equation*}

According to the proof in \cite{Ju2}, we can verify the oblique boundary condition.
\begin{lemma}\label{lem3.5}$($See Lemma 3.1 in \cite{OC}$)$
    If $u$ is a  smooth uniformly convex solution of $\eqref{e1.1.3}$ and $\eqref{e1.1.4}$ , and then the boundary condition is oblique, i.e.,
    \begin{align}\label{oblique equation}
        \langle\beta, \nu\rangle \geq0,\quad x\in\partial \Omega,
    \end{align}
\end{lemma}

Let $\omega=\langle \beta,\nu\rangle+\tau h(Du)$,
where $\tau$ is a positive constant to be determined.
Then we have following  lemma. 
\begin{lemma}\label{vnestimate}
	 Suppose $x_0\in \partial\Omega$  such that
	 $\omega$ reaches the minimum on $\partial\Omega$ at $x_0$.  If $f\in \mathscr{A}$ \eqref{oscf} 
	and $u$ is a strictly convex solution to \eqref{e3.3.4}, then
	\begin{equation}\label{vn}
		\frac{\partial \omega}{\partial \nu}(x_0)\geq-C_1,
	\end{equation}
	where $C_1$ is a constant depending only on $n, \Omega, \tilde{\Omega}$ and $f$.
\end{lemma}
\begin{proof}
	By rotation, we may assume that $x_0=0$ and $\nu(x_0)=(0,1)=:e_n$. Denote a neighborhood of $x_0$ in $\Omega$ by
	$$\Omega_{r}:=\Omega\cap B_{r}(x_{0}),$$
	where $0<r<1$ is a positive constant such that $\nu$ is well defined in $\Omega_{r}$. In order to obtain the desired results, it suffices to consider the auxiliary function
	$$\Phi(x)=\omega(x)-\omega(x_{0})+\sigma\ln(1+k\tilde{h}(x))+A|x-x_{0}|^{2},$$
	where $\sigma$, $k$ and $A$ are positive constants to be determined.

	At first, we estimate $\mathcal{L} \omega$, where $\mathcal{L}=G_{ij}\partial_{ij}+(G_{p_i}-f_{p_i})\partial_{i}$,  by Corollary \ref{c3.3}, Lemma \ref{l3.30} and $\mathcal{L}u_k=f_k$, we have
	\begin{equation*}\label{e3.10}
		\begin{aligned}
			\mathcal{L}\omega=&G_{ij}u_{il}u_{jm}(h_{p_{k}p_{l}p_{m}}\nu_{k}+\tau h_{p_{l}p_{m}})+2G_{ij}h_{p_{k}p_{l}}u_{li}\nu_{kj}\\
			&+h_{p_{k}}\mathcal{L}\nu_{k}+\tau h_{p_{k}}\mathcal{L}u_k+h_{p_{k}p_{l}}\nu_{k}\mathcal{L}u_l\\
			\leq& (h_{p_{k}p_{l}p_{m}}\nu_{k}+\tau h_{p_{l}p_{m}}+\delta_{lm})G_{ij}u_{il}u_{jm}+C_{2}\mathcal{T}_{G}+C_{3},
		\end{aligned}
	\end{equation*}
	 and
	$$2G_{ij}h_{p_{k}p_{l}}u_{li}\nu_{kj}\leq  G_{ij}u_{im}u_{mj}+C_{2}\mathcal{T}_{G}.$$
	
	Since $D^{2}h\leq-\theta I$, we may choose $\tau$ large enough depending on the known data such that
	$$(h_{p_{k}p_{l}p_{m}}\nu_{k}+\tau h_{p_{l}p_{m}}+\delta_{lm})<0.$$
	
	Consequently, we deduce that
	\begin{equation}\label{e3.11}
		\mathcal{L}\omega\leq C_{4}\mathcal{T}_{G} \    \ in \   \ \Omega,
	\end{equation}
	by the convexity of $u$.

	By noting that $\tilde{h}$ is the defining function of $\Omega$, $G_{ij}$ has positive lower bound and $G_{p_i}-f_{p_i}$ is bounded, we show that
	\begin{equation}\label{e3.12}
		\begin{aligned}
			\mathcal{L}(\ln(1+k\tilde{h}))&=G_{ij}\left(\frac{k\tilde{h}_{ij}}{1+k\tilde{h}}
			-\frac{k\tilde{h}_{i}}{1+k\tilde{h}}\frac{k\tilde{h}_{j}}{1+k\tilde{h}}\right)
			+(G_{p_i}-f_{p_i})\frac{k\tilde{h}_{i}}{1+k\tilde{h}}\\
			&\triangleq G_{ij}\frac{k\tilde{h}_{ij}}{1+k\tilde{h}}-G_{ij}\eta_{i}\eta_{j}+(G_{p_i}-f_{p_i})\eta_{i}\\
			&\leq\left(-\frac{k\tilde{\theta}}{1+k\tilde{h}}+C_{5}-C_{6}|\eta-C_{7}I|^{2}\right)\mathcal{T}_{G}\\
			&\leq\left(-\frac{k\tilde{\theta}}{1+k\tilde{h}}+C_{5}\right)\mathcal{T}_{G},
		\end{aligned}
	\end{equation}
	where $\eta=\left(\frac{k\tilde{h}_{1}}{1+k\tilde{h}}, \frac{k\tilde{h}_{2}}{1+k\tilde{h}},\cdots, \frac{k\bar{h}_{n}}{1+k\tilde{h}}\right)$.
	
	By taking $r$ to be small enough, we have
	\begin{equation}\label{e3.13}
		\begin{aligned}
			0\leq \tilde{h}(x)&=\tilde{h}(x)-\tilde{h}(x_{0})\\
			&\leq \sup_{\Omega_{r}}|D\tilde{h}||x-x_{0}|\\
			&\leq r\sup_{\Omega}|D\tilde{h}|\leq \frac{\tilde{\theta}}{3C_{5}}.
		\end{aligned}
	\end{equation}
	
	By choosing $k=\frac{7C_{5}}{\tilde{\theta}}$ and applying (\ref{e3.13}) to (\ref{e3.12}), we obtain
	\begin{equation}\label{e3.14}
		\mathcal{L}(\ln(1+k\tilde{h}))
		\leq -C_{5}\mathcal{T}_{G}.
	\end{equation}
	
	Combining (\ref{e3.11}) with (\ref{e3.14}), a direct computation yields
	\begin{equation*}
		\mathcal{L}(\Phi(x))\leq (C_{4}-\sigma C_{5}+2A+C_{6})\mathcal{T}_{G},
	\end{equation*}
  where $C_{6}=\frac{2A}{\Lambda_6}\max_{\overline{\Omega}\times \overline{\tilde{\Omega}}}\left(|G_{p_i}|+|f_{p_i}|\right)$.	By taking
	$$\sigma=\frac{C_{4}+2A+C_{6}}{C_{5}},$$
	
	\begin{equation}\label{e3.15}
		\mathcal{L}\Phi\leq 0,\,\,  x\in\Omega_{r}.
	\end{equation}
	
	It is clear that $\Phi(x)\geq0$ on $\partial\Omega$.
	Because $\omega$ is bounded, then it follows that we can choose $A$ large enough depending on the known data such that on $\Omega\cap\partial B_{r}(x_{0})$,
	\begin{equation*}
		\begin{aligned}
			\Phi(x)&=\omega(x)-\omega(x_0)+\sigma\ln(1+k\tilde{h}(x))+Ar^{2}\\
			&\geq\omega(x)-\omega (x_0)+Ar^{2}\geq 0.
		\end{aligned}
	\end{equation*}
	The above argument implies that
	\begin{equation}\label{e3.15aaa}
		\Phi(x)\geq 0, \,\, \quad \text{for} \quad x \in \partial\Omega_{r}.
	\end{equation}
	
	According to the maximum principle, we get that
	\begin{equation}\label{e3.15aaaa}
		\Phi(x)|_{\Omega_{r}}\geq 0.
	\end{equation}
	By the above inequality and $\Phi(x_0)=0$, we have $\partial_n\Phi(x_0)\geq 0$, which gives the desired estimate \eqref{vn}.
\end{proof}

The unit inward normal vector of $\partial\Omega$ can be expressed by $\nu=D\tilde{h}$. For the same reason, $\tilde{\nu}=Dh$, where $\tilde{\nu}=(\tilde{\nu}_{1}, \tilde{\nu}_{2},\cdots,\tilde{\nu}_{n})$ is the unit inward normal vector of $\partial\tilde{\Omega}$.  $\tilde{h}$ is the defining function of $\Omega$ . That is,
$$\Omega=\{\tilde{p}\in\mathbb{R}^{n} : \tilde{h}(\tilde{p})>0\},\ \ \ |D\tilde{h}|_{\partial\Omega}=1, \ \ \ D^2\tilde{h}\leq -\tilde{\theta}I,$$
where $\tilde{\theta}$ is some positive constant.

By using \eqref{transfoemation G} and noting that $\tilde{h}$ is the defining function of $\Omega$, then we know $\tilde{u}$ satisfies
\begin{align}\label{e2.17}
	\begin{cases}
		\tilde{G}(y,D^{2}\tilde{u})=-f(D\tilde{u},\tilde{x})-c,\quad\tilde{x}\in \tilde{\Omega}, \\
		\qquad\tilde{h}(D\tilde{u})=0,\quad \tilde{x}\in\partial\tilde{\Omega},\\
	\end{cases}
\end{align}
where $\tilde{G}(y,D^{2}\tilde{u})=-G(y,D^{2}\tilde{u}^{-1})$.

Let \begin{align*}
	\tilde{\beta}=(\tilde{\beta}^{1}, \cdots, \tilde{\beta}^{n}),\quad \tilde{\beta}^{k}:=\tilde{h}_{\tilde{p}_{k}}(D\tilde{u}),
\end{align*}
using the representation as the works of \cite{HRY} and \cite{OK},
we also define
$$\tilde{\omega}=\langle\tilde{\beta}, \tilde{\nu}\rangle+\tilde{\tau} \tilde{h}(D\tilde{u}),$$
in which
$$\langle\tilde{\beta}, \tilde{\nu}\rangle=\langle\beta, \nu\rangle,$$
and $\tilde{\tau}$ is a  positive constant to be determined. 

Then we can get 

\begin{lemma}\label{tildevnestimate}
	 suppose $\tilde{\omega}$ reaches the minimum on $\partial\tilde{\Omega}$ at $y_0$.  Let $f\in \mathscr{A}$, satisfying  \eqref{oscf} and
	\begin{equation}\label{Dxf}
		|D_{x}f|\leq \frac{\theta\Lambda_6}{2\max_{\overline{\tilde{\Omega}}}|Dh|}.
	\end{equation}
	If $\tilde{u}$ is a strictly convex solution to \eqref{e2.17}, then
	\begin{equation}\label{tildevn}
		\frac{\partial \tilde{\omega}}{\partial \tilde{\nu}}(y_0)\geq-C_8,
	\end{equation}
	where $C_8$ is a constant depending only on $n, \Omega, \tilde{\Omega}$ and $f$.
\end{lemma}
\begin{proof}
	By rotation, we may also assume that $y_0=0$ and $\tilde{\nu}(y_0)=(0,1)=:e_n$. Denote a neighborhood of $y_0$ in $\tilde{\Omega}$ by 
	$$\tilde{\Omega}_{\rho}:=\tilde{\Omega}\cap B_{\rho}(y_{0}),$$
	where $\rho$ is a positive constant such that $\tilde{\nu}$ is well defined in $\tilde{\Omega}_{\rho}$. In order to obtain the desired results,
	we  consider the  auxiliary function
	$$\Psi(y)=\tilde{\omega}(y)-\tilde{\omega}(y_{0})+\tilde{k}h(y) +\tilde{A}|y-y_{0}|^{2},$$
	where $\tilde{k}$ and $\tilde{A}$ are positive constants to be determined. 
	
		Firstly, we turn to  estimate $\mathcal{\tilde{L}} \tilde{\omega}$, where
	the elliptic operator
	$$\mathcal{\tilde{L}}=\tilde{G}_{ij}\partial_{p_ip_j}+f_i\partial_{p_i}.$$
	By calculation, we arrive at $\mathcal{\tilde{L}}\tilde{u}_k=-f_{p_k}-\tilde{G}_{y_{k}}$ and
	\begin{equation*}
		\begin{aligned}
			\mathcal{\tilde{L}}\tilde{\omega}=&\tilde{G}_{ij}\tilde{u}_{li}\tilde{u}_{mj}(\tilde{h}_{\tilde{p}_{k}\tilde{p}_{l}\tilde{p}_{m}}\tilde{\nu}_{k}+\tilde{\tau} \tilde{h}_{\tilde{p}_{l}\tilde{p}_{m}})+2\tilde{G}_{ij}\tilde{h}_{\tilde{p}_{k}\tilde{p}_{l}}\tilde{u}_{li}\tilde{\nu}_{kj}\\
			&-(f_{p_k}+\tilde{G}_{y_{k}})(\tilde{h}_{\tilde{p}_{k}\tilde{p}_{m}}\tilde{\nu}_{m}+\tilde{\tau} \tilde{h}_{\tilde{p}_{k}})+\tilde{h}_{\tilde{p}_k}\mathcal{\tilde{L}}\tilde{\nu}_{k}\\
			\leq&(\tilde{h}_{\tilde{p}_{k}\tilde{p}_{l}\tilde{p}_{m}}\tilde{\nu}_{k}+\tilde{\tau} \tilde{h}_{\tilde{p}_{l}\tilde{p}_{m}}+\delta_{lm})\tilde{G}_{ij}\tilde{u}_{il}\tilde{u}_{jm}+C_{9}\mathcal{T}_{\tilde{G}}+C_{10}(1+\tilde{\tau}),
		\end{aligned}
	\end{equation*}
	where
	$$2\tilde{G}_{ij}\tilde{h}_{\tilde{p}_{k}\tilde{p}_{l}}\tilde{u}_{li}\tilde{\nu}_{kj}\leq \delta_{lm}\tilde{G}_{ij}\tilde{u}_{il}\tilde{u}_{jm}
	+C_{9}\mathcal{T}_{\tilde{G}},$$
	by Lemma 3.6 in \cite{HR1}. Since $D^{2}\tilde{h}\leq-\tilde{\theta}I$, we only need to choose $\tilde{\tau}$ sufficiently large depending on the known data such that
	$$\tilde{h}_{\tilde{p}_{k}\tilde{p}_{l}\tilde{p}_{m}}\tilde{\nu}_{k}+\tilde{\tau} \tilde{h}_{\tilde{p}_{l}\tilde{p}_{m}}+\delta_{lm}<0.$$
	Therefore,
	\begin{equation}\label{e3.16}
		\mathcal{\tilde{L}}\tilde{\omega}\leq C_{11}\mathcal{T}_{\tilde{G}},
	\end{equation}
	by the convexity of $\tilde{u}$.
	
	By the uniformy concavity if $h$ and \eqref{Dxf}, one yields that
	\begin{align*}
		\mathcal{\tilde{L}}(h(y))&=\tilde{G}_{ij}h_{ij}+f_{i}h_i\\
		&\leq -\theta\mathcal{T}_{\tilde{G}}+f_{i}h_{i}\\
		&\leq -\frac{\theta}{2}\mathcal{T}_{\tilde{G}},
	\end{align*}
	Then, it follows from \eqref{e3.16} that
	$$\mathcal{\tilde{L}}\Psi\leq(C_{11}-\frac{\tilde{k}\theta}{2}+2\tilde{A}+C_{12})\mathcal{T}_{\tilde{G}},$$
	where $C_{12}=\frac{2A}{\Lambda_8}\max_{\overline{\Omega}\times \overline{\tilde{\Omega}}}|f_{i}|$. Let $$\tilde{k}=\frac{4\tilde{A}+2C_{11}+2C_{12}}{\tilde{k}\theta},$$
	
	It is easy to check that $\Psi(y)\geq0$ on $\partial\tilde{\Omega}$. Note that $\tilde{\omega}$ is bounded, it follows that we can choose $\tilde{A}$ large enough depending on the known data, such that on $\tilde{\Omega}\cap\partial B_{\rho}(y_{0})$,
	\begin{equation*}
		\Psi (y)=\tilde{\omega}(y)-\tilde{\omega}(y_{0})+\tilde{k}h(y)+\tilde{A}\rho^{2}\geq \tilde{\omega}(y)-\tilde{\omega}(y_{0})
		+\tilde{A}\rho^{2}\geq 0.
	\end{equation*}

	The rest of the proof of \eqref{tildevn} is the same as \eqref{vn}. Thus the proof of \eqref{tildevn} is completed.
\end{proof}

By the above two Lemmas, we have the uniformly oblique estimate 
\begin{proposition}\label{llll3.4}
  Assume that $\Omega$, $\tilde{\Omega}$ are bounded, uniformly convex domains with smooth boundary in $\mathbb{R}^{n}$ and  $\tilde{\Omega}\subset\subset B_1(0)$. Suppose f satisfiy \eqref{oscf} and \eqref{Dxf}
   If $u$ is a strictly convex solution to \eqref{e3.3.4}, then the strict obliqueness estimate
\begin{equation}\label{e3.3.12}
\langle\beta, \nu\rangle\geq \frac{1}{C_{13}}>0
\end{equation}
holds on $\partial \Omega$ for some universal constant $C_{13}$, which depends only on $n,\Omega,$ $\tilde{\Omega}$ and $f$.
\end{proposition}
\begin{proof}

Let $\omega$ and $x_0\in\partial\Omega$ as the same in the proof of Lemma \ref{vnestimate}. 
Combined the
 convexity of $\Omega$, we extend $\nu$ smoothly to a tubular neighborhood of $\partial\Omega$ such that in the matrix sense
\begin{equation}\label{e3.3.9}
  \left(\nu_{kl}\right):=\left(D_k\nu_l\right)\leq -\frac{1}{C_{14}}\operatorname{diag} (\underbrace {1,\cdots, 1}_{n-1},0),
\end{equation}
where $C_{14}$ is a positive constant depending only on $\Omega$. By Lemma \ref{lem3.5}, we see that $h_{p_{n}}(Du(x_0))\geq0$.

At the minimum point  $x_0$, it yields
\begin{equation}\label{e3.3.14}
 0=\omega_l=h_{p_np_k}u_{kl}+h_{p_k}\nu_{kl}+\tau h_{p_k}u_{kl}, \quad 1\leq l\leq n-1.
\end{equation}
From Lemma \eqref{vnestimate}, it follows that  
\begin{equation}\label{e3.1.1}
\omega_n(x_0)>-C_{1}
\end{equation}
holds. We observe that \eqref{e3.1.1} can be rewritten as
\begin{equation}\label{e3.7}
h_{p_np_k}u_{kn}+h_{p_k}\nu_{kn}+ \tau h_{p_k}u_{kn}>-C_{1}.
\end{equation}

Multiplying $h_{p_n}$ on both sides of \eqref{e3.7} and $h_{p_r}$ on both sides of \eqref{e3.3.14} respectively, and summing up together, one gets
\begin{align*}
 \tau h_{p_k}h_{p_l}u_{kl}\geq -C_{1}h_{p_n}- h_{p_k}h_{p_l}\nu_{kl}- h_{p_k}h_{p_np_l}u_{kl}.
\end{align*}

Combining (\ref{e3.3.9}) with
$$ 1\leq l\leq n-1,\quad h_{p_k}u_{kl}=\frac{\partial h(Du)}{\partial x_l}=0,\quad h_{p_k}u_{kn}=\frac{\partial h(Du)}{\partial x_n}\geq 0,\quad -h_{p_np_n}\geq 0,$$
we have
$$\tau h_{p_k}h_{p_l}u_{kl}\geq-C_{1}h_{p_n}+\frac{1}{C_{14}}|Dh|^2-\frac{1}{C_{14}}h^2_{p_n}
\geq-C_{15}h_{p_n}+\frac{1}{C_{15}}-\frac{1}{C_{15}}h^2_{p_n},$$
where $C_{15}=\max\{C_{1},C_{14}\}$.
Now, to obtain the estimate $\langle \beta,\nu\rangle$, we divide $-C_{15}h_{p_n}+\frac{1}{C_{15}}-\frac{1}{C_{15}}h^2_{p_n}$ into two cases at $x_0$.

Case (i).  If
$$-C_{15}h_{p_n}+\frac{1}{C_{15}}-\frac{1}{C_{15}}h^2_{p_n}\leq \frac{1}{2C_{15}},$$
then
$$h_{p_k}(Du)\nu_{k}=h_{p_n}\geq \sqrt{\frac{1}{2}+\frac{C^{4}_{15}}{4}}-\frac{C^{2}_{15}}{2}.$$
It means that there is a uniform positive lower bound for $\underset{\partial\Omega}\min \langle \beta,\nu\rangle$.

Case (ii). If
$$-C_{15}h_{p_n}+\frac{1}{C_{15}}-\frac{1}{C_{15}}h^2_{p_n}> \frac{1}{2C_{15}},$$
then we know that there is a positive lower bound for $h_{p_k}h_{p_l}u_{kl}$.

Denote $y_{0}=Du(x_0)$, then $$\tilde{\omega}(y_{0})=\omega(x_{0})=\min_{\partial\tilde{\Omega}} \tilde{\omega}.$$

From Lemma \eqref{tildevnestimate}, it follows that  
\begin{equation}\label{e3.9}
\tilde{\omega}_{n}(y_{0})\geq -C_{8}
\end{equation}
holds. Then, we obtain the positive lower bounds of $\tilde{h}_{\tilde{p}_{k}}\tilde{h}_{\tilde{p}_{l}}\tilde{u}_{kl}$ or
$$h_{p_{k}}(Du)\nu_{k}=\tilde{h}_{\tilde{p}_{k}}(D\tilde{u})\tilde{\nu}_{k}=\tilde{h}_{\tilde{p}_{n}}\geq
\sqrt{\frac{1}{2}+\frac{C^{4}_{16}}{4}}-\frac{C^{2}_{16}}{2}.$$

On the other hand, it is easy to check that
$$\tilde{h}_{\tilde{p}_{k}}\tilde{h}_{\tilde{p}_{l}}\tilde{u}_{kl}=\nu_{i}\nu_{j}u^{ij}.$$

Then by the positive lower bounds of $h_{p_{k}}h_{p_{l}}u_{kl}$ and $\tilde{h}_{\tilde{p}_{k}}\tilde{h}_{\tilde{p}_{l}}\tilde{u}_{kl}$, the desired conclusion can be obtained by
\begin{align*}
\langle \beta,\nu\rangle=\sqrt{h_{p_k}h_{p_l}u_{kl}u^{ij}\nu_i\nu_j}.
\end{align*}

For details of the proof of the above formula, the readers can refer to \cite{Ju2} or \cite{J2}.
It remains to prove the key estimates \eqref{e3.1.1} and \eqref{e3.9}.
\end{proof}

\section{The global $C^2$ estimate}\label{sec4}
We now proceed to carry out the $C^2$ estimate. The strategy is to construct suitable barrier function and use elliptic
maximum principle to reduce the $ C^2$ global estimates
of $u$ and $\tilde{u}$ to the boundary.
\begin{lemma}\label{lem4.1}
 Assume that $\Omega$, $\tilde{\Omega}$ are bounded, uniformly convex domains with smooth boundary in $\mathbb{R}^{n}$, $\tilde{\Omega}\subset\subset B_{1}(0)$. Suppose $f$ satisfy \eqref{oscf} and   $0<\alpha<1$, $u\in C^{2+\alpha}(\bar{\Omega})$ is uniformly convex solution to \eqref{e1.1.3} and \eqref{e1.1.4}, then there exists  positive constants $C_{17}$,$C_{18}$,$C_{19}$ depending only on $n$, $\Omega$, $\tilde{\Omega}$ and $f$, such that
\begin{equation}\label{eq3.15}
D^2 u(x) \leq C_{17} I_n,\ \ x\in\bar\Omega
\end{equation}
and
\begin{align}
\label{e4.4.2}
    C_{18}\leq u_{11}+u_{22}+\cdots+u_{nn}\leq C_{19},
\end{align}
where $I_n$ is the $n\times n$ identity matrix.
\end{lemma}
\begin{proof}
By the proof of \eqref{e2.2.8} in Lemma \ref{L2.1}, we can show that
$F=\sum_{i=1}^na_{ii}$ is bounded.
From the formula in (\ref{e2.2.13}), we get that
$$u_{kl}=v\sum_{i}^nb_{ik}a_{im}b_{ml},\quad k,l=1,2,\cdots,n.$$
Then, by using \begin{equation*}
b_{ij}=\delta_{ij}-\frac{D_{i}uD_{j}u}{1+v},\;
|Du|<1,\;
v=\sqrt{1-|Du|^2}
\end{equation*}
and the second boundary condition,   we obtain $D^2 u(x) $ is bounded.

In addition, by Lemma \ref{L2.1}, we know
\begin{align*}
    \Lambda_1\leq\kappa_1+\cdots +\kappa_n\leq \Lambda_2
\end{align*}
and 
then we obtain \eqref{e4.4.2}.
\end{proof}

The following Lemma is to reduce the global $C^2$ estimates of $\tilde{u}$ to the boundary.

\begin{lemma}\label{lem4.2}
     Suppose $f \in \mathcal{A}$ and  satisfy \eqref{oscf}, \eqref{Dxf} and
     \begin{equation}\label{Dxpf}
     	\max_{\overline{\Omega}\times \overline{\tilde{\Omega}}}|D_{xp}f|\leq \frac{\Lambda_6\theta}{8\max_{\overline{\tilde{\Omega}}}|h|}.
     \end{equation}
       If $\tilde{u}$ is a smooth uniformly convex solution of \eqref{e2.17} and there hold \eqref{e2.2.8}-\eqref{e2.2.12}, then there exists a positive constant $C_{20}$ depending only on $n, \Omega, \tilde{\Omega}$ and $f$, such that
       \begin{equation}\label{intobod}
       	\sup _{\tilde{\Omega}}\left|D^2 \tilde{u}\right| \leq 2\max _{\partial \tilde{\Omega}}\left|D^2 \tilde{u}\right|+C_{20}.
       \end{equation}
\end{lemma}
\begin{proof}
	Let 
	$$\tilde{L}=\tilde{G}_{ij}\partial_{p_ip_j}+f_i\partial_{p_i},$$

Denote
\begin{align*}
    s_{ij}=\frac{1}{\sqrt{1-|y|^2}}\left(\delta_{ij}+\frac{y_{i}y_{j}}{1-|y|^{2}}\right),\quad[\tilde{u}_{ij}]=[u_{ij}]^{-1}.
\end{align*}
And then
\begin{align*}
    \frac{1}{\sqrt{1-|y|^2}} I_n \leq [s_{ij}]\leq \frac{1}{\left(1-|y|^2\right)^{\frac{3}{2}}}I_n,
\end{align*}
where $I_n$ is the $n \times n$ identity matrix. By calculating, we show that
 \begin{align*}
   G(y,D^2u)&=\mathrm{div} \left(\frac{Du}{\sqrt{1-|Du|^2}}\right)\\
   &=\frac{1}{\sqrt{1-|Du|^2}}\left(\delta_{ij}+\frac{y_{i}y_{j}}{1-|y|^{2}}\right)u_{ij},
\end{align*}
and then
\begin{equation}\label{GeqS}
	\tilde{G}(y, D^{2}\tilde{u})=-s_{ij}u_{ij}.
\end{equation}
By the rotation of the coordinate system for any fixed point $y_0\in \tilde{\Omega}$, such that  we can get at $y_0=Du(x_0)$,
\begin{align*}
	\left[D^2{\tilde{u}}\right]_{|y_0}=\mathrm{diag}(\tilde{u}_{11},\tilde{u}_{22},\cdots\tilde{u}_{nn})
\end{align*}
and
\begin{align*}
\left[D^2u\right]_{|y_0}=\mathrm{diag}(u_{11},u_{22},\cdots,u_{nn}).
\end{align*}
For any $1\leq k\leq n$, by \eqref{e4.4.2} we obtain
\begin{align*}
	-\tilde{G}_{ij,rs}\tilde{u}_{ijk}\tilde{u}_{rsk}-2\tilde{G}_{ij,y_k}\tilde{u}_{ijk}&=-G_{ii,ii}\tilde{u}_{iik}^2-2\tilde{G}_{ii,y_k}\tilde{u}_{iik}\\
	&=s_{ii}\frac{2}{\tilde{u}_{ii}^3}\tilde{u}_{iik}^2-\frac{2}{\tilde{u}_{ii}^2}\dfrac{\partial s_{ii}}{\partial y_k}\tilde{u}_{iik}\\
	&\geq s_{ii}\frac{2}{\tilde{u}_{ii}^3}\tilde{u}_{iik}^2-C_{21}\frac{2}{\tilde{u}_{ii}^2}|\tilde{u}_{iik}|\\
	&\geq \frac{2}{\tilde{u}_{ii}^3}\tilde{u}_{iik}^2-\frac{2}{\tilde{u}_{ii}^3}\tilde{u}_{iik}^2-C_{21}^2\frac{2}{\tilde{u}_{ii}}\\
	&=-C_{21}^2\frac{2}{\tilde{u}_{ii}}
	\geq -2C_{21}^2C_{19},
\end{align*}
where $C_{21}=\max_{\overline{\tilde{\Omega}}}|\frac{y_k+2y_i\delta_{ik}}{(1+|y|^2)^{3/2}}+\frac{3y_{i}^2y_{k}}{(1-|y|^2)^{5/2}}|$. 

 The second derivation of $y_k$ with respect to \eqref{e2.17}, we can obtain
\begin{align*}
	\tilde{\mathcal{L}}\tilde{u}_{kk}=-\tilde{G}_{ij,rs}\tilde{u}_{ijk}\tilde{u}_{rsk}-2\tilde{G}_{ij,y_k}\tilde{u}_{ijk}-f_{ij}\tilde{u}_{ik}\tilde{u}_{jk}-2f_{ip_k}\tilde{u}_{ik}-f_{p_kp_k}
\end{align*}
Then by the concavity of $f$ in $x$, we have
\begin{align*}
\tilde{\mathcal{L}}\tilde{u}_{kk}&\geq -2C_{21}^2C_{19}-2f_{ip_k}\tilde{u}_{ik}-f_{p_kp_k}\\
&\geq -2\max_{\overline{\Omega}\times\overline{\tilde{\Omega}}}|D_{xp}f|\sup_{\tilde{\Omega}}|D^2\tilde{u}|-f_{p_kp_k}-2C_{21}^2C_{19}.
\end{align*}
Let
$$w=\max_{\partial\tilde{\Omega}} \tilde{u}_{kk}+\frac{2}{\Lambda_8\theta}\left(2C_{21}^2C_{19}+\max_{\overline{\Omega}\times\overline{\tilde{\Omega}}}|D_{pp}f|+2\max_{\overline{\Omega}\times\overline{\tilde{\Omega}}}|D_{xp}f|\sup_{\tilde{\Omega}}|D^2\tilde{u}|\right)h.$$
By direct calculation and \eqref{e3..3.3}, \eqref{Dxf}, we obtain
\begin{align*}
	\tilde{L}w&\leq \frac{2}{\Lambda_8\theta}\left(2C_{21}^2C_{19}+\max_{\overline{\Omega}\times\overline{\tilde{\Omega}}}|D_{pp}f|+2\max_{\overline{\Omega}\times\overline{\tilde{\Omega}}}|D_{xp}f|\sup_{\tilde{\Omega}}|D^2\tilde{u}|\right)\left(-\theta\mathcal{T}_{\tilde{G}}+f_{i}h_{p_i}\right)\\
	&\leq -2C_{21}^2C_{19}-\max_{\overline{\Omega}\times\overline{\tilde{\Omega}}}|D_{pp}f|-\max_{\overline{\Omega}\times\overline{\tilde{\Omega}}}|D_{xp}f|\sup_{\tilde{\Omega}}|D^2\tilde{u}|\quad in~\tilde{\Omega}
\end{align*}
and thus
$$\tilde{L}(w-\tilde{u}_{kk})\leq 0\quad in~\tilde{\Omega}.$$
It is obvious that $w-\tilde{u}_{kk}\geq 0$ on $\partial\tilde{\Omega}$. Then, by the maximum principle, we obtain
\begin{align*}
	\sup_{\tilde{\Omega}}\tilde{u}_{kk}\leq \sup_{\tilde{\Omega}} w&\leq \max_{\partial\tilde{\Omega}} \tilde{u}_{kk}+\max_{\overline{\tilde{\Omega}}}|h|\frac{2}{\Lambda_8\theta}\left(2C_{21}^2C_{19}+\max_{\overline{\Omega}\times\overline{\tilde{\Omega}}}|D_{pp}f|+2\max_{\overline{\Omega}\times\overline{\tilde{\Omega}}}|D_{xp}f|\sup_{\tilde{\Omega}}|D^2\tilde{u}|\right)\\
	&\leq \max_{\partial\tilde{\Omega}}|D^2\tilde{u}|+\frac{2}{\Lambda_8\theta}\max_{\overline{\tilde{\Omega}}}|h|\left(2C_{21}^2C_{19}+\max_{\overline{\Omega}\times\overline{\tilde{\Omega}}}|D_{pp}f|\right)\\
	&\quad +\frac{4}{\Lambda_8\theta}\max_{\overline{\tilde{\Omega}}}|h|\max_{\overline{\Omega}\times\overline{\tilde{\Omega}}}|D_{xp}f|\sup_{\tilde{\Omega}}|D^2\tilde{u}|.
\end{align*}
Let $C_{20}=\frac{4}{\Lambda_8\theta}\max_{\overline{\tilde{\Omega}}}|h|\left(2C_{21}^2C_{19}+\max_{\overline{\Omega}\times\overline{\tilde{\Omega}}}|D_{pp}f|\right)$, this completes the proof.
\end{proof}

Recalling that $\tilde{\beta}=(\tilde{\beta}^{1}, \cdots, \tilde{\beta}^{n})$ with $\tilde{\beta}^{k}:=\tilde{h}_{\tilde{p}_{k}}(D\tilde{u})$ and $\tilde{\nu}=(\tilde{\nu}_{1}, \tilde{\nu}_{2},\cdots,\tilde{\nu}_{n})$ is the unit inward normal vector of $\partial\tilde{\Omega}$.
In the following we give the arguments, the readers can see there for more details.
For any tangential direction $\tilde \varsigma$, we have
\begin{equation}\label{e3.35}
   \tilde{u}_{\tilde\beta \tilde\varsigma}=\tilde h_{\tilde{p}_k}(D\tilde u)\tilde u_{k\tilde\varsigma}=0.
\end{equation}

Then the second order derivative of $\tilde u$ on the boundary is also controlled by $u_{\tilde\beta \tilde\varsigma}$, $u_{\tilde\beta \tilde\beta}$ and $u_{\tilde\varsigma\tilde\varsigma}$. At $\tilde x\in \partial\tilde\Omega$, any unit vector $\tilde\xi$ can be written in terms of a tangential component $\tilde\varsigma(\tilde\xi)$ and a component in the direction $\tilde\beta$ by
$$\tilde\xi=\tilde\varsigma(\tilde\xi)+\frac{\langle \tilde\nu,\tilde\xi\rangle}{\langle\tilde\beta,\tilde\nu\rangle}\tilde\beta,$$
where
$$\tilde\varsigma(\tilde\xi):=\tilde\xi-\langle \tilde\nu,\tilde\xi\rangle \tilde\nu-\frac{\langle \tilde\nu,\tilde\xi\rangle}{\langle\tilde\beta,\tilde\nu\rangle}\tilde\beta^T$$
and
$$\tilde\beta^T:=\tilde\beta-\langle \tilde\beta,\tilde\nu\rangle \tilde\nu.$$

We observe that $\langle\tilde\beta,\tilde\nu\rangle=\langle\beta,\nu\rangle$.
By the uniformly obliqueness estimate (\ref{e3.3.12}), we have
\begin{equation}\label{e3.36}
\begin{aligned}
|\tilde{\varsigma}(\tilde{\xi})|^{2}&=1-\left(1-\frac{|\tilde{\beta}^{T}|^{2}}{\langle\tilde{\beta},\tilde{\nu}\rangle^{2}}\right)
\langle\tilde{\nu},\tilde{\xi}\rangle^{2}
-2\langle\tilde{\nu},\tilde{\xi}\rangle\frac{\langle\tilde{\beta}^{T},\tilde{\xi}\rangle}{\langle\tilde{\beta},\tilde{\nu}\rangle}\\
&\leq C_{22}.
\end{aligned}
\end{equation}

Denote $\varsigma:=\frac{\varsigma(\xi)}{|\varsigma(\xi)|}$, then by (\ref{e3.3.12}), (\ref{e3.35}) and (\ref{e3.36}), we arrive at
\begin{equation}\label{e4.4.5}
\begin{aligned}
\tilde{u}_{\tilde{\xi}\tilde{\xi}}&=|\tilde{\varsigma}(\tilde{\xi})|^{2}
\tilde{u}_{\tilde{\varsigma}\tilde{\varsigma}}+2|\tilde{\varsigma}(\tilde{\xi})|\frac{\langle\tilde{\nu},\tilde{\xi}\rangle}{\langle\tilde{\beta},\tilde{\nu}\rangle}\tilde{u}_{\tilde{\beta}\tilde{\varsigma}}+
\frac{\langle\nu,\xi\rangle^{2}}{\langle\beta,\nu\rangle^{2}}
\tilde{u}_{\tilde{\beta}\tilde{\beta}}\\
&=|\tilde{\varsigma}(\tilde{\xi})|^{2}\tilde{u}_{\tilde{\varsigma}\tilde{\varsigma}}+\frac{\langle\tilde{\nu},\tilde{\xi}\rangle^{2}}{\langle\tilde{\beta},\tilde{\nu}\rangle^{2}}
\tilde{u}_{\tilde{\beta}\tilde{\beta}}\\
&\leq C_{23}(\tilde{u}_{\tilde{\varsigma}\tilde{\varsigma}}+\tilde{u}_{\tilde{\beta}\tilde{\beta}}).
\end{aligned}
\end{equation}

Therefore, we also only need to estimate $\tilde{u}_{\tilde\beta\tilde\beta}$ and $\tilde{u}_{\tilde\varsigma\tilde\varsigma}$ respectively.
\begin{lemma}\label{lem4.3}
If $\tilde{u}$ is a strictly convex solution of \eqref{e2.17} and $f$ satisfy \eqref{oscf}, then there exists a positive constant $C_{24}$ depending only on $\Omega$, $\tilde{\Omega}$ and $f$, such that
\begin{equation}\label{e4 4.6}
   \max_{\partial\Omega}\tilde{u}_{\tilde{\beta}\tilde{\beta}} \leq C_{24}.
\end{equation}
\end{lemma}
\begin{proof}
Let $\tilde{x}_0\in\partial\tilde{\Omega}$, such that $\tilde{u}_{\tilde{\beta}\tilde{\beta}}(\tilde{x}_0)=\max_{\partial\Omega}\tilde{u}_{\tilde{\beta}\tilde{\beta}}$.
To estimate the upper bound of $\tilde{u}_{\tilde{\beta}\tilde{\beta}}$,
we consider the barrier function
$$\tilde\Psi(y):=-\tilde{h}(D\tilde{u}(y))+C_0h(y).$$

 Combining the uniformly concavity of $h$ with \eqref{Dxf},  we get 
\begin{align*}
	\mathcal{\tilde{L}}(h)&=\tilde{G}_{ij}h_{ij}+f_ih_{i}\leq -\theta\mathcal{T}_{\tilde{G}}+f_ih_{y}\leq -\frac{\theta}{2}\mathcal{T}_{\tilde{G}}.
\end{align*}

Using the equations (\ref{e2.17}), a direct computation shows that
\begin{equation}\label{e3.39}
\begin{aligned}
\mathcal{\tilde{L}}\left(-\tilde{h}(D\tilde{u})\right)&=\tilde{G}_{ij}\left(-\tilde{h}_{\tilde{p}_{k}\tilde{p}_{l}}
\partial_{ki}\tilde{u}\partial_{lj}\tilde{u}\right)-\tilde{h}_{\tilde{p}_{k}}(\tilde{G}_{y_{k}}+f_{p_k})\\
&\leq C_{25}\mathcal{T}_{\tilde{G}},
\end{aligned}
\end{equation}
where we use the estimates (\ref{e3..3.3})-(\ref{e3..3.4}) in Corollary \ref{c3.4}.
Therefore, we obtain
$$\tilde {\mathcal{L}}\tilde\Psi(y)\leq \left(C_{25}-\frac{C_0\theta}{2}\right)\mathcal{T}_{\tilde{G}}. $$
Let \begin{align*}
C_0=\frac{2C_{25}}{\theta},
\end{align*}
we get $\tilde {\mathcal{L}}\tilde\Psi(y)\leq 0~\text{in}~\tilde{\Omega}$.

For any $y\in \partial\tilde{\Omega}$, $D\tilde{u}(y)\in \partial\Omega$, then $\tilde{h}(D\tilde{u})=0$. It is clear that $h=0$ on $\partial\tilde{\Omega}$. 
Applying the maximum principle, we get
$$\tilde\Psi(y)\geq 0,\quad\quad y\in \tilde{\Omega}.$$
Combining it with $\tilde\Psi(\tilde{x}_0)=0$, we obtain $\tilde\Psi_{\tilde{\beta}}(\tilde{x}_0)\geq 0$, which implies
$$\frac{\partial \tilde{h}}{\partial \tilde{\beta}}(D\tilde{u}(\tilde{x}_0))\leq C_0.$$
On the other hand, we see that at $\tilde{x}_0$,
$$\frac{\partial \tilde{h}}{\partial \tilde{\beta}}=\langle D\tilde{h}(D\tilde{u}),\tilde{\beta}\rangle=\frac{\partial \tilde{h}}{\partial p_k}\tilde{u}_{kl}\tilde{\beta}^l=\tilde{\beta}^k\tilde{u}_{kl}\tilde{\beta}^l=\tilde{u}_{\tilde{\beta}\tilde{\beta}}.$$
Therefore, letting $C_{24}= C_{0}$ we have
$$\tilde{u}_{\tilde{\beta}\tilde{\beta}}=\frac{\partial \tilde{h}}{\partial \tilde{\beta}}\leq C_{24}.$$
\end{proof}

To estimate the  double tangential derivative of $\tilde{u}$, we need the following Lemma, which proof is simimalr to \eqref{e3.16}.
\begin{lemma}\label{Lphistimate}
	If $\tilde{u}$
	is a strictly convex solution to \eqref{e2.17} and $f$ satisfy \eqref{oscf}, \eqref{Dxf}, respectively. Fix a smooth function $\eta: \tilde{\Omega} \times \Omega  \to \mathbb{R}$  defined by $\varphi=\eta(\tilde{x},
	D\tilde{u})$, it follows that
	$$\tilde{\mathcal{L}}\varphi\leq C_{26}\mathcal{{T}}_{\tilde{G}},$$
	where $C_{26}$ depends only on $n,  \Omega, \tilde{\Omega},f$ and $\|\eta\|_{C^2(\overline{\Omega}\times \overline{\tilde{\Omega}})}$.
\end{lemma}
Then, we can get
\begin{lemma}\label{lem4.5}
If $\tilde{u}$ is a strictly convex solution of \eqref{e2.17} and $f \in \mathcal{A}$ satisfy \eqref{oscf}, \eqref{Dxf}, \eqref{Dxpf}, then there exists a positive constant $C_{27}$ depending only on $u_0$, $\Omega$, $\tilde{\Omega}$ and $f$, such that
\begin{equation*}\label{e4.43}
   \max_{\partial\tilde{\Omega}} \tilde{u}_{\tilde\varsigma\tilde\varsigma} \leq C_{27}.
\end{equation*}
\end{lemma}
\begin{proof}
We assume that $\tilde{x}_{0}\in\partial\Omega$, $e_{n}$ is the unit inward normal vector of $\partial\tilde{\Omega}$ at $\tilde{x}_0$ and $e_{1}$ is the tangential vector of $\partial\tilde{\Omega}$ at $\tilde{x}_0$ respectively, such that
$$ \max_{\partial\tilde{\Omega}} \tilde{u}_{\tilde\varsigma\tilde\varsigma}=\tilde{u}_{11}(\tilde{x}_0)=:\mathcal{ M}.$$
For any $y\in \partial\tilde{\Omega}$, it follows from the proof of (\ref{e3.36}) and (\ref{e4.4.5}) that
\begin{equation}\label{e4.45}
\begin{aligned}
\tilde u_{\tilde\xi\tilde\xi}&=|\tilde\varsigma(\tilde\xi)|^2\tilde u_{\tilde\varsigma\tilde\varsigma}+ \frac{\langle \tilde\nu,\tilde\xi\rangle^2}{\langle \tilde\beta,\tilde\nu\rangle^2}\tilde u_{\tilde\beta\tilde\beta}\\
          &\leq \left(1+C_{28}\langle \tilde\nu,\tilde\xi\rangle^2-2\langle \tilde\nu,\tilde\xi\rangle \frac{\langle \tilde\beta^T,\tilde\xi\rangle}{\langle \tilde\beta,\tilde\nu\rangle}\right) \mathcal{M}
                 + \frac{\langle \tilde\nu,\tilde\xi\rangle^2}{\langle \tilde\beta,\tilde\nu\rangle^2}\tilde u_{\tilde\beta\tilde\beta}.
\end{aligned}
\end{equation}

Without loss of generality, we assume that $\mathcal{M} \geq 1$. Thus by (\ref{e3.3.12}), (\ref{e4 4.6}) and (\ref{e4.45}) we deduce that
\begin{equation*}\label{eq3.9a}
  \frac{\tilde u_{\tilde\xi\tilde\xi}}{\mathcal {M}}+2\langle \tilde\nu,\tilde\xi\rangle \frac{\langle \tilde\beta^T,\tilde\xi\rangle}{\langle \tilde\beta,\tilde\nu\rangle}
      \leq 1+C_{28}\langle \tilde\nu,\tilde\xi\rangle^2.
\end{equation*}
Let $\tilde\xi=e_1$, then
\begin{equation*}\label{eq3.10a}
  \frac{\tilde u_{11}}{\mathcal{ M}}+2\langle \tilde\nu,e_1\rangle \frac{\langle \tilde\beta^T,e_1\rangle}{\langle \tilde\beta,\tilde\nu\rangle}
             \leq 1+C_{28}\langle \tilde\nu,e_1\rangle^2.
\end{equation*}
We see that the function
\begin{equation*}\label{eq3.11a}
\tilde w:=A|y-\tilde x_0|^2-\frac{\tilde u_{11}}{\mathcal{ M}}-2\langle \tilde\nu,e_1\rangle \frac{\langle \tilde\beta^T,e_1\rangle}{\langle \tilde\beta,\tilde\nu\rangle}+C_{28}\langle \tilde\nu,e_1\rangle^2+1
\end{equation*}
satisfies
$$\tilde w|_{\partial\tilde\Omega}\geq 0,\quad  \tilde w(\tilde x_0)=0.$$
Denote a neighborhood of $\tilde{x}_0$ in $\tilde\Omega$ by
$$\tilde\Omega_{r}:=\tilde\Omega\cap B_{r}(\tilde{x}_0),$$
where $0<r<1$ is a positive constant such that $\tilde\nu$ is well defined in $\tilde\Omega_{r}$.
Let us consider
$$-2\langle \tilde\nu,e_1\rangle \frac{\langle \tilde\beta^T,e_1\rangle}{\langle \tilde\beta,\tilde\nu\rangle}+C_{28}\langle \tilde\nu,e_1\rangle^2+1$$
as a known function depending on $y$ and $D\tilde u$. Then by the proof of (\ref{e3.39}), we also obtain
\begin{equation*}
\tilde{\mathcal{L}} \left(-2\langle \tilde\nu,e_1\rangle \frac{\langle \tilde\beta^T,e_1\rangle}{\langle \tilde\beta,\tilde\nu\rangle}+C_{28}\langle \tilde\nu,e_1\rangle^2+1\right)\leq C_{29}\mathcal{T}_{\tilde{G}}.
\end{equation*}
It follows from  in Lemmas \ref{lem4.2}, \ref{lem4.3}, \eqref{e4.4.5} and the concavity of $f$ in $x$ that
\begin{align*}
\tilde{\mathcal{L}}\tilde{u}_{11}&=-\tilde{G}_{ij,rs}\tilde{u}_{ij1}\tilde{u}_{rs1}-2\tilde{G}_{ij,y_1}\tilde{u}_{ij1}-f_{ij}\tilde{u}_{i1}\tilde{u}_{j1}-2f_{ip_1}\tilde{u}_{i1}-f_{p_1p_1}\\
&\geq -2C_{21}^2C_{19}-2\max_{\overline{\Omega}\times\overline{\tilde{\Omega}}}|D_{xp}f|\sup_{\tilde{\Omega}}|D^2\tilde{u}|-\max_{\overline{\Omega}\times\overline{\tilde{\Omega}}}|D_{pp}f|\\
&\geq -2\max_{\overline{\Omega}\times\overline{\tilde{\Omega}}}|D_{xp}f|\left(C_{20}+2C_{23}(C_{24}+\mathcal{M})\right)-C_{30}\\
&\geq -C_{31}\mathcal{ M}\mathcal{T}_{\tilde{G}},
\end{align*}
where $C_{30}=2C_{21}^2C_{19}+\max_{\overline{\Omega}\times\overline{\tilde{\Omega}}}|D_{pp}f|$.
We set
$$\tilde\Upsilon:=\tilde w+\tilde{C_0}h.$$
Furthermore, by(\ref{e3.3.12}), (\ref{e3.3.14}), (\ref{e4.4.5}) and (\ref{e4 4.6}),  we can choose the constant $A$ large enough such that
$$\tilde w|_{\tilde\Omega \cap \partial B_{r}(\tilde x_0)} \geq 0. $$
As in the proof in Lemma \ref{lem4.3},
\begin{align*}
    \tilde{\mathcal{L}}\Upsilon
   & \leq 2A\mathcal{T}_{\tilde{G}}-2Af_i(y_i-x_{0i})-\frac{\tilde{L}\tilde{u}_{11}}{\mathcal{ M}}+C_{29}\mathcal{T}_{\tilde{G}}-\frac{\tilde{C_0}\theta}{2}\mathcal{T}_{\tilde{G}}\\
   &\leq  2A\mathcal{T}_{\tilde{G}}-2Af_i(y_i-x_{0i})+C_{31}\mathcal{T}_{\tilde{G}}+C_{29}\mathcal{T}_{\tilde{G}}-\frac{\tilde{C_0}\theta}{2}\mathcal{T}_{\tilde{G}}\\
   &\leq \left(2A+C_{32}+C_{31}+C_{29}-\frac{\tilde{C_0}\theta}{2}\right)\mathcal{T}_{\tilde{G}},
\end{align*}
 where we use \eqref{Dxf} for the first inequality and $C_{32}=\frac{2A}{\Lambda_8}\max_{\overline{\Omega}\times \overline{\tilde{\Omega}}}|D_xf|$. Then, by choosing $\tilde{C_0}=\frac{4A+2C_{32}+2C_{31}+2C_{29}}{\theta}$, we get that
 \begin{align*}
     \mathcal{\tilde{L}}\tilde\Upsilon\leq 0,\quad y\in \tilde\Omega_{r}.
 \end{align*}

A standard barrier argument makes conclusion of
$$\tilde\Upsilon_{\tilde\beta}(\tilde x_0)\geq0.$$
Therefore,
\begin{equation}\label{eq3.12a}
 \tilde u_{11\tilde\beta}(\tilde x_0)\leq C_{33}\mathcal{ M}.
\end{equation}

On the other hand, differentiating $\tilde h(D\tilde u)$ twice in the direction $e_1$ at $\tilde x_0$, we have
$$\tilde h_{p_k}\tilde u_{k11}+\tilde h_{p_kp_l}\tilde u_{k1}\tilde u_{l1}=0.$$
The concavity of $\tilde h$ yields that
$$\tilde h_{p_k}\tilde u_{k11}=-\tilde h_{p_kp_l}\tilde u_{k1}\tilde u_{l1}\geq \tilde\theta \mathcal{M}^2.$$
Combining it with $\tilde h_{p_k}\tilde u_{k11}=\tilde u_{11\tilde\beta}$ and using (\ref{eq3.12a}), we obtain
$$\tilde\theta  \mathcal{M}^2\leq C_{33}\mathcal{M}.$$
Then we get the upper bound of $\mathcal{M}=\tilde u_{11}(\tilde x_0)$ and thus the desired result follows.
\end{proof}
By Lemma \ref{lem4.5}, Lemma \ref{lem4.5} and (\ref{e4.4.5}), we obtain the $C^2$ a-priori estimate of $\tilde{u}$ on the boundary.
\begin{lemma}\label{lem4.6}
If $\tilde{u}$ is a strictly convex solution of \eqref{e2.17} and $f \in \mathcal{A}$ satisfy \eqref{oscf}, \eqref{Dxf}, \eqref{Dxpf}, then there exists a positive constant $C_{34}$ depending only on $n$, $\Omega$, $\tilde{\Omega}$ and $f$, such that
\begin{equation*}\label{eq3.13a}
\max_{\partial\tilde\Omega}|D^2\tilde u| \leq C_{34}.
\end{equation*}
\end{lemma}

By Lemma \ref{lem4.3} and Lemma \ref{lem4.6}, we can see that
\begin{lemma}\label{lem4.7}
If $\tilde{u}$ is a strictly convex solution of \eqref{e2.17} and $f \in \mathcal{A}$ satisfy \eqref{oscf}, \eqref{Dxf}, \eqref{Dxpf}, then there exists a positive constant $C_{35}$ depending only on $n$, $\Omega$, $\tilde{\Omega}$ and $f$, such that
\begin{equation*}\label{eq3.14a}
\max_{\bar{\tilde\Omega}}|D^2\tilde u| \leq C_{35}.
\end{equation*}
\end{lemma}

By Lemma \ref{lem4.1} and Lemma \ref{lem4.7}, we conclude that
\begin{lemma}\label{lem3.6}
Assume that $\Omega$, $\tilde{\Omega}$ are bounded, uniformly convex domains with smooth boundary in $\mathbb{R}^{n}$ and $\tilde{\Omega}\subset\subset B_{1}(0)$. If $u$ is a strictly convex solution to \eqref{e1.1.3}-\eqref{e1.1.4} and $f \in \mathcal{A}$ satisfy \eqref{oscf}, \eqref{Dxf}, \eqref{Dxpf}, then there exists a positive constant $C_{36}$ depending only on $n$, $\Omega$, $\tilde{\Omega}$ and $f$, such that
\begin{equation*}\label{eq4.4.8}
\frac{1}{C_{36}}I_n\leq D^2 u(x) \leq C_{36} I_n,\; x\in\bar\Omega,
\end{equation*}
where $I_n$ is the $n\times n$ identity matrix.
\end{lemma}

\vspace{3mm}
\section{Existence and uniqueness of solutions}\label{sec5}
In this section, we prove the existence of a solution to (\ref{e1.1.3}) and (\ref{e1.1.4}) by the continuity method
and some tricks which we learn from Wang-Huang-Bao's work \cite{Wang2023}.
The proof of uniqueness up to a constant, we refer the readers to Lemma 5.1 in Huang-Li \cite{RS}.

\noindent\textbf{ Proof of Theorem \ref{t1.1}.}

 For $t\in [0,1]$, consider the following problem
\begin{align}\label{eq6.1}
\begin{cases}
      G(Du,D^2u)=tf(x,Du)+c(t) ,\quad &x\in\Omega,\\
  \qquad\;  h(Du)=0,\quad  &x\in \partial\Omega.
\end{cases}
\end{align}
By the main result in \cite{HW}, \eqref{eq6.1} is solvable if $t=0$. We denote the Banach space
$$\mathcal{X}=\{u\in C^{2,\alpha}(\overline{\Omega}):\int_\Omega u=0\}$$
and 
$$\mathcal{Y}:=C^\alpha(\bar{\Omega})\times C^{1,\alpha}(\partial\Omega).$$
Define a map from $\mathcal{F}: \mathcal{X}\times \mathbb{R} \to \mathcal{Y}$ as
$$\mathcal{F}(u,c)=\left(G(Du,D^2u) - tf(x,Du) - c(t), h(Du)\right).$$
Therefore, if $(u,c)\in\mathcal{X}\times\mathbb{R}$ is a solution of \eqref{eq6.1}, and then $\mathcal{F}(u,c)=(0,0)$. The linearized operator $D\mathcal{F}:\mathcal{X}\times\mathbb{R}\to\mathcal{Y}$ by
$$D\mathcal{F}_{(u,c)}^t(w,a)=(\mathcal{L}_tw-a,\mathcal{N}w).$$
where the operator $\mathcal{L}_t:C^{2,\alpha}(\overline{\Omega})\rightarrow C^{\alpha}(\overline{\Omega})$ is define by
$$\mathcal{L}_tw(x)=G_{ij}(Du,D^2u)\partial_{ij}w+(G_{p_i}-tf_{p_i})(Du,D^2u)\partial_iw,$$
for $x\in \Omega$. Moreover, the operator $\mathcal{N}:C^{2,\alpha}(\overline{\Omega})\rightarrow C^{1,\alpha}(\partial\Omega)$ is defined by
$$\mathcal{N}w(x) =\langle D w(x),D\tilde{h}(D u(x))\rangle$$
for $x\in\partial\Omega$. Obviously, $\mathcal{L}$ is an elliptic operator, the boundary condition is oblique.
Repeating the proof of Proposition 3.1 in \cite{Brendle2008ABV} and Proposition 5.1 in \cite{HW}, we know that $D\mathcal{F}_{(u,c)}^t$ is invertible if $c(t)$ is bounded for any $t \in [0, 1]$ and $(u, c)$ being the solution to \eqref{eq6.1}.

Define the set
\begin{align*}
    I:=\{t\in[0,1]: \eqref{eq6.1} \text{has at least one convex solution}\}.
\end{align*}
By Huang-Wei-Ye's theorem \cite{HW}, $0\in I$ and $I$ is a non-empty.
We claim that $I=[0,1]$, which is equivalent to prove that the set $I$ is not only open, but also closed.
It follows from Proposition 3.1 in \cite{Brendle2008ABV} and Theorem 17.6 in \cite{GT} that $I$ is open.
We next use the a-priori estimates in Section \ref{sec3} and section \ref{sec4} to prove that $I$
is a closed subset of $[0, 1]$.
It is equivalent to the fact that for any monotone increasing  sequence $\{t_k \}\subset I$, if $\lim_{k\rightarrow \infty}t_k = t_0$, then $t_0 \in I$.

For each $t_k$, we denote $(u_k,c(t_k))$ solving problem \ref{eq6.1}
\begin{equation*}
\begin{cases}
      G(Du_k,D^2u_k)=t_kf(x,Du_k)+c(t_k) ,\quad &x\in\Omega,\\
\qquad\;\;h(Du_k)=0,\quad &x\in \partial\Omega.
\end{cases}
\end{equation*}
Combining Lemma \ref{lem3.6} with Evans-Krylov theory, we can prove that
\begin{equation*}
  \| u_k \|_{C^{2,\alpha}(\bar{\Omega})} \leq C_{42},
\end{equation*}
where $C_{42}$ is independent of $k$.
Since by using Lemma \ref{lem3.6} again, we have $$|c(t_k)|=|G(Du_{t_k},D^2u_{k})-t_kf(x,Du_k)| \leq \Lambda_2+ \max_{\overline{\Omega}\times \overline{\tilde{\Omega}}}|f(x,p)|,$$
and hence by using Arzela-Ascoli Theorem, we know that there exists $\hat{u} \in C^{2,\alpha}(\overline{\Omega})$,
$\hat{c} \in \R$ and a subsequence of $\{t_k\}$,
which is still denoted by $\{t_k\}$, such that by letting $k \rightarrow \infty$ to obtain
\begin{equation*}
\begin{cases}
     \| u_k-\hat u\|_{C^2(\overline{\Omega})}\rightarrow 0,\\
   c(t_k)\rightarrow \hat{c}.
\end{cases}
\end{equation*}
Letting $k\rightarrow \infty$, we deduce that
\begin{equation*}
\begin{cases}
      G(D\hat{u},D^2\hat{u})=t_0f(x,D\hat{u})+\hat{c} ,\quad &x\in\Omega,\\
   \qquad h(D\hat{u})=0,\quad &x\in \partial\Omega.
\end{cases}
\end{equation*}
Therefore $t_0 \in I$ and thus $I$ is closed.
Consequently, $I=[0,1]$ and by Schauder estimate, $u\in C^{4,\alpha}(\overline{\Omega})$, 
we complete the proof of Theorem \ref{t1.1}.
\vspace{5mm}

\noindent{\bf Acknowledgments:} We are very grateful to the editors for their very professional handling of our paper.

\vspace{5mm}
\noindent{\bf Funding:} Jiguang Bao is supported by the National Key Research and Development Program of China (No. 2020YFA0712904) and the National Natural Science Foundation of China (No. 12371200).  Rongli Huang is supported by the National Natural Science Foundation of China (No. 12101145).


\vspace{5mm}



 \begin{thebibliography}{DU}
 	
 	\bibitem{Bart1984}  Bartnik, R. (1984). Existence of maximal surfaces in asymptotically flat spacetimes.
 	Comm. Math. Phys.  94: 155--175. doi: projecteuclid.org/euclid.cmp/1103941280
 	
 	\bibitem{Cho1979}  Bruhat, Y.~C., Fischer, A.~E.,  Marsden, J.E. (1979). Maximal hypersurfaces and
 		positivity of mass. Isolated gravitating systems in general relativity.  396--456.
 		
 		\bibitem{Mard1980}   Marsden, J.~E.,  Tipler, F.~J. (1980).  Maximal hypersurfaces and foliations of constant mean curvature in general relativity. phys. Rep.  66: no.~3, 109--139. doi: 10.1016/0370-1573(80)90154-4.
 			
 		\bibitem{Bayard}  Bayard, P. (2006).  Entire spacelike hypersurfaces of prescribed scalar curvature in Minkowski space. Calc. Var. Partial Differential Equations.  26: no.~2, 245--264.	doi: 10.1007/s00526-005-0367-0.
 			
 		\bibitem{LLJ}  Caffarelli, L.,  Nirenberg, L.,  Spruck, J. (1986).  Nonlinear second order elliptic equations. IV. Starshaped compact Weingarten hypersurfaces.   Current topics in partial differential equations, Kinokuniya, Tokyo.  1--26.
 			
 			
 		\bibitem{Choi}  Choi, H.~I.,  Treibergs, A. (1990). Gauss maps of spacelike constant mean curvature hypersurfaces of Minkowski space. J. Differential Geom. 32(3): 775--817. doi: 10.1515/CRELLE.2009.009.
 				
 		\bibitem{TAE}   Treibergs, A.~E. (1982). Entire spacelike hypersurfaces of constant mean curvature in Minkowski space. Invent. Math. 66(1): 39--56. doi: 10.1090/proc/12969.
 				
 				
 		\bibitem{yau-sch 1979}  Schoen, R.~M., Yau, S.-T. (1979). On the proof of the positive mass conjecture in general relativity.  Comm. Math. Phys.  65: no.~1, 45--76. doi: projecteuclid.org/euclid.cmp/1103904790.
 				
 		\bibitem{Bart1982}  Bartnik, R.~A.,  Simon, L.~M. (1982). Spacelike hypersurfaces with prescribed boundary values and mean curvature. Comm. Math. Phys.  87: no.~1, 131--152. doi: projecteuclid.org/euclid.cmp/1103921909
 			
 				
 		\bibitem{Bay2003}    Bayard, P. (2003). Dirichlet problem for space-like hypersurfaces
 		with prescribed scalar curvature in $R^{n,1}$. Calc. Var. Partial Differential Equations.  18: no.~1, 1--30. doi: 10.1007/s00526-002-0178-5.
 			
 		\bibitem{Bere2013}  Bereanu, C., Jebelean,  P., Torres, P.~J. (2013). Multiple positive radial solutions for a Dirichlet problem involving the mean curvature operator in Minkowski space. J. Funct. Anal. 265: no.~4, 644--659. doi: 10.1016/j.jfa.2013.04.006.
 				
 				\bibitem{Mon 1999}	 Montiel, S. (1999). Uniqueness of spacelike hypersurfaces of constant mean curvature in foliated spacetimes. Math. Ann. 314: no.~3, 529--553. doi: 10.1007/s002080050306.
 			
 			\bibitem{Aquino1}  Aquino C.~P.,  de~Lima H.~F. (2014).  On the umbilicity of complete constant mean curvature spacelike hypersurfaces. Math. Ann.  360: no.~3-4, 555--569. doi: 10.1007/s00208-014-1049-z
 			
 			\bibitem{Bon2023}  Bonsante, F., Seppi, A.,  Smillie, P. (2023). Complete CMC hypersurfaces in Minkowski $(n+1)$-space. Comm. Anal. Geom. 31: no.~4, 799--845. doi: 10.4310/cag.2023.v31.n4.a2.
 			
 			\bibitem{Cheng1976}  Cheng, S.~Y., Yau, S.-T. (1976). Maximal space-like hypersurfaces in the Lorentz-Minkowski spaces. Ann. of Math. (2)  104: no.~3, 407--419. doi: 10.2307/1970963.
 			
 			\bibitem{Xin0}  Xin, Y.~L. (1991).  On the Gauss image of a spacelike hypersurface with constant mean curvature in Minkowski space. Comment. Math. Helv. 66: no.~4, 590--598. doi: 10.1007/BF02566667.
 			
 			\bibitem{Raf2007}  L\'opez~Camino, R. (2007).  On the existence of spacelike constant mean curvature surfaces spanning two circular contours in Minkowski space. J. Geom. Phys.  57: no.~11, 2178--2186. doi: 10.1016/j.geomphys.2007.06.006
 			
 			\bibitem{Fern2007}  Fern\'andez~Delgado, I.,  L\'opez~Fern\'andez, F.~J. (2007).  Periodic maximal surfaces in the Lorentz-Minkowski space $\Bbb L^3$. Math. Z. 256: no.~3, 573--601. doi: 10.1007/s00209-006-0087-y.
 			
 			\bibitem{Akm2024}  Akamine, S.,  Fujino, H. (2024).  Duality of boundary value problems for minimal and maximal surfaces. Comm. Anal. Geom.  32: no.~4, 1057--1094. doi: 10.4310/cag.241015230035.
 			
 			
 			\bibitem{Maw2016} Mawhin, J.~L.,  Torres, P.~J. (2016). Prescribed mean curvature graphs with Neumann boundary conditions in some FLRW spacetimes. J. Differential Equations. 261: no.~12, 7145--7156. doi: 10.1016/j.jde.2016.09.013.
 	
 	        \bibitem{J2}  Urbas, J.~I.~E. (2002).  Weingarten hypersurfaces with prescribed gradient image.  Math. Z. 240: no.~1, 53--82. doi: 10.1007/s002090100362.
 	 
 	        \bibitem{Ju3}  Urbas, J.~I.~E. (2007).  A remark on minimal Lagrangian diffeomorphisms and the Monge-Amp\`{e}re equation. Bull. Austral. Math. Soc. 76: no.~2, 215--218. doi: 10.1017/S0004972700039605.
 	

            \bibitem{urbas1}  Urbas, J.~I.~E. (1995).  Nonlinear oblique boundary value problems for Hessian equations in two dimensions Ann. Inst. H. Poincar\'e{} C Anal. Non Lin\'eaire. 12: no.~5, 507--575. doi: 10.1016/S0294-1449(16)30150-0.
            
            \bibitem{Ju2}  Urbas, J.~I.~E. (1997).  On the second boundary value problem for equations of Monge-Amp\`{e}re type. J. Reine Angew. Math. 487: 115--124. doi: 10.1515/crll.1997.487.115.

            \bibitem{Ju1}   Urbas, J.~I.~E. (2001). The second boundary value problem for a class of Hessian equations.  Comm. Partial Differential Equations. 26: no.~5-6, 859--882. doi: 10.1081/PDE-100002381.
            
            \bibitem{Nessi}  von~Nessi, G.~T. (2010). On the second boundary value problem for a class of modified-Hessian equations.
            Comm. Partial Differential Equations.  35: no.~5, 745--785. doi: 10.1080/03605301003632317.
            
            \bibitem{Jiang1}  Jiang, F.,  Trudinger, N.~S. (2018). On the second boundary value problem for Monge-Amp\`ere type equations and geometric optics. Arch. Ration. Mech. Anal.  229: no.~2, 547--567.
            doi: 10.1007/s00205-018-1222-8.
            
            \bibitem{Chen1}  Chen, S.,  Liu, J.,  Wang, X.~J. (2021). Global regularity for the Monge-Amp\`{e}re equation with natural boundary condition.
            Ann. of Math. (2) 194: no.~3, 745--793. doi: 10.4007/annals.2021.194.3.4.
            
            \bibitem{Savin}  Savin, O.~V.,  Yu, H. (2020).  Global $W^{2,1+\epsilon}$ estimates for Monge-Amp\`{e}re equation with natural boundary condition. J. Math. Pures Appl. (9)  137: 275--289.
            doi: 10.1016/j.matpur.2019.09.006.
            
            \bibitem{Brendle2008ABV}
            Brendle, S., Warren.  M.~W. (2010).
            A boundary value problem for minimal lagrangian graphs. 
            J. Differential Geom.  84: no.~2, 267--287. doi: projecteuclid.org/euclid.jdg/1274707314.
            
            \bibitem{Huang2014OnTS}
             Huang, R.~L. (2015).
             On the second boundary value problem for {L}agrangian mean curvature flow.
            J. Funct. Anal. 269: no.~4, 1095--1114. doi: 10.1016/j.jfa.2015.05.003.
            
            \bibitem{Wang2023}  Wang, C.,  Huang, R.~L.,  Bao, J.~G. (2023).  On the second boundary value problem for Lagrangian mean curvature equation.
            Calc. Var. Partial Differential Equations.  62: no.~3, Paper No. 74, 30 pp. doi: 10.1007/s00526-022-02412-3.
            
            \bibitem{Wang2024}
            C. Wang, R.~L. Huang and J.~G. Bao,
            {\it  On the second boundary value problem for a class of fully nonlinear
            	flow {III}}.
            J. Evol. Equ. {\bf 24} (2024), no.~3, Paper No. 52, 38 pp. doi: 10.1007/s00028-024-00983-6.
            
            \bibitem{Jiangbao2025}  Jiang, Q.~F.,  Bao, J.~ G.  On the second boundary value problem for Lagrangian	mean curvature type equation. preprint. 
            
            \bibitem{HR1}  Huang, R.~L. (2024). On the second boundary value problem for mean curvature flow in Minkowski space.   arXiv preprint: arXiv:2404.05972.
            
   
            \bibitem{RS} Huang, R.~L., Li, S. (2022).  On the second boundary value problem for special Lagrangian curvature potential equation.
            Math. Z.  302: No.~1, 391--417. doi: 10.1007/s00209-022-03060-1.
            
            \bibitem{ou} Huang, R.~L.,  Ou, Q.~Z. (2017). On the second boundary value problem for a class of fully nonlinear equations.
            J. Geom. Anal.  27: no.~4, 2601--2617. doi: 10.1007/s12220-017-9774-7.
            
            \bibitem{HOWY} Huang, R.~L.,  Qu, C.~Z.,  Wang, Z.~Z.,  Wo, W.~F. (2024).  Hessian curvature hypersurfaces with prescribed Gauss image  arXiv preprint: arXiv:2412.09159.
            
            \bibitem{HW} Huang, R.~L.,  Wei, D.~Y.,  Ye,  Y.~H. (2024)  The constant mean curvature hypersurfaces with prescribed gradient image. arXiv preprint: arXiv:2411.00817.
            
             \bibitem{ma}  Ma, X.~N.  (1999). A necessary condition of solvability for the capillarity boundary of Monge-Amp\`ere equations in two dimensions. Proc. Amer. Math. Soc.  127: no.~3, 763--769. doi: 10.1090/S0002-9939-99-04750-4.
             
            \bibitem{urbas2}  Lions, P.-L.,  Trudinger, N.~S.,  Urbas, J.~I.~E. (1986).  The Neumann problem for equations of Monge-Amp\`ere type. Comm. Pure Appl. Math. {\bf 39} , no.~4, 539--563. doi:10.1002/cpa.3160390405.
             
            
            
            \bibitem{Li}  Li, A.~M. (1995). Spacelike hypersurfaces with constant Gauss-Kronecker curvature in the Minkowski space. Arch. Math. (Basel)  64: no.~6, 534--551. doi: 10.1007/BF01195136.
            
            \bibitem{EH}  Ecker K.,  Huisken, G. (1991). Parabolic methods for the construction of spacelike slices of prescribed mean curvature in cosmological spacetimes.
            Comm. Math. Phys.  135: No.~3, 595--613. doi: projecteuclid.org/euclid.cmp/1104202145.
            
            \bibitem{Ec}  Ecker, K. (1997). Interior estimates and longtime solutions for mean curvature flow of noncompact	spacelike hypersurfaces in Minkowski space. J. Differential Geom.  46: no.~3, 481--498. doi: projecteuclid.org/euclid.jdg/1214459975.
            	
           \bibitem{JS}   Spruck, J. (2003)   Geometric Aspects of the theory of fully nonlinear elliptic equations.  Clay Math. Proc. 2, Amer. Math. Soc., Providence, RI. 
            	
            \bibitem{OC}  Schn\"urer, O.~C. (2002). Translating solutions to the second boundary value problem
            		for curvature flows. Manuscripta Math.  108: no.~3, 319--347. doi: 10.1007/s002290200265.
            		
            \bibitem{HRY}  Huang, R.~L.,  Ye, Y.~H. (2019).  On the second boundary value problem for a class of fully nonlinear flows I.  Int. Math. Res. Not. 18: 5539--5576. doi: 10.1093/imrn/rnx278.
            
            \bibitem{OK} Schn\"urer, O.~C.,  Smoczyk, K. (2003).  Neumann and second boundary value problems for Hessian and  Gauss curvature flows.
            Ann. Inst. H. Poincar\'e{} C Anal. Non Lin\'eaire  20: no.~6, 1043--1073. doi: 10.1016/S0294-1449(03)00021-0.
            
            \bibitem{GT}  Gilbarg, D.,  Trudinger, N.~S. (1983).  Elliptic partial differential equations of second order, second edition. 
            Grundlehren der mathematischen Wissenschaften, 224, Springer, Berlin, doi: 10.1007/978-3-642-61798-0.
            
            
            
           

   
	
	
	

\end{thebibliography}
\end{document}